\documentclass[oneside,english]{amsart}
\usepackage[T1]{fontenc}
\usepackage[latin9]{inputenc}
\usepackage{geometry}
\usepackage{amsthm}
\usepackage{hyperref}
\usepackage{tikz}
\usepackage{kbordermatrix}
\usepackage{float}
 \usepackage{setspace}

\makeatletter
\numberwithin{equation}{section}
\numberwithin{figure}{section}
\theoremstyle{plain}
\newtheorem{thm}{\protect\theoremname}
\theoremstyle{plain}

\theoremstyle{plain}

\theoremstyle{plain}
\newtheorem{lem}[thm]{\protect\lemmaname}
\theoremstyle{plain}

\theoremstyle{plain}
\newtheorem{cor}[thm]{\protect\corollaryname}
\theoremstyle{plain}
\newtheorem{definition}[thm]{\protect\definitionname}
\theoremstyle{plain}

\makeatother

\usepackage{babel}
\providecommand{\conjecturename}{Conjecture}
\providecommand{\lemmaname}{Lemma}
\providecommand{\theoremname}{Theorem}
\providecommand{\problemname}{Problem}
\providecommand{\propositionname}{Proposition}
\providecommand{\corollaryname}{Corollary}
\providecommand{\definitionname}{Definition}
\providecommand{\examplename}{Example}
\begin{document}

\title{The Inverse Eigenvalue Problem for Linear Trees}

\author{Charles R. Johnson$^\dag$ and Tanay Wakhare$^\ast$}

\thanks{{\scriptsize
\hskip -0.4 true cm MSC(2010): Primary: 05C50; Secondary: 15B57.
$^\dag$Corresponding author, 
}}

\address{$^\ast$~Department of Electrical Engineering and Computer Science, MIT, Cambridge, MA 02139}
\email{twakhare@mit.edu}
\address{$^\dag$~Department of Mathematics, College of William and Mary, Williamsburg, VA 23185}
\email{crjohn@wm.edu}

\maketitle

\begin{abstract}
We prove the sufficiency of the Linear Superposition Principle for linear trees, which characterizes the spectra achievable by a real symmetric matrix whose underlying graph is a linear tree. The necessity was previously proven in \cite{TwoAndrews}. This is the most general class of trees for which the inverse eigenvalue problem has been solved. We explore many consequences, including the Degree Conjecture for possible spectra, upper bounds for the minimum number of eigenvalues of multiplicity $1$, and the equality of the diameter of a linear tree and its minimum number of distinct eigenvalues.
\end{abstract}

\section{Introduction}

Let $G$ be an undirected graph on $n$ vertices, and denote by $\mathcal{S}(G)$ the set of all $n$-by-$n$ real symmetric matrices, the graph of whose nonzero off-diagonal entries is $G$. By convention, $G$ places no restriction on the diagonal entries of $A \in \mathcal{S}(G)$, other than that they be real. Each element of $\mathcal{S}(G)$ has an ordered list of eigenvalues, including multiplicities, and, thus, $\mathcal{S}(G)$ exhibits a catalog of multiplicity lists occurring among the matrices in $\mathcal{S}(G)$. These may be presented as unordered lists (partitions of $n$) or as ordered lists that respect the order of the underlying eigenvalues. The former are denoted by $\mathcal{L}(G)$ and the latter as $\mathcal{L}_0(G)$. It has long been a goal to understand these lists as a function of $G$, i.e., to understand how the structure of $G$ limits the multiplicities of the eigenvalues of matrices in $\mathcal{S}(G)$. For several reasons, interest has focused upon the case in which $G=T$, a tree. The problem of multiplicities is still far from settled there. The recent book \cite{Johnson3} describes the background and most known works on this subject (including generalizations beyond symmetric matrices and trees). See this reference for any terminology or notation not defined herein.


\begin{center}
\begin{figure}[!ht]
\centering
\def\r{1.9}
\def\k{6pt}
\begin{tikzpicture}[scale=0.7,
rel/.style={circle, draw, only marks, mark=*, mark options={fill=red},mark size=\k},
bor/.style={rectangle,draw,rounded corners=0.6ex, minimum size=4pt, inner sep=5pt},
coo/.style={circle, draw, only marks, mark=*, fill=red, minimum size=\k+3pt},
]
\foreach \x/\xtext in {1,2,3}{
\draw (0pt,0pt) -- ({\r*cos(\x*2*pi/3 r)}, {\r*sin(\x*2*pi/3 r)});}
\draw plot[only marks, mark=*, mark options={fill=red},mark size=\k] 
coordinates{(0,0)} node[] {};
\foreach \y/\xtext in {1,2,3}{
\draw[shift={({\r*cos(2*pi/3 r)}, {\r*sin(2*pi/3 r)})}]
(0pt,0pt) -- ++({\r*cos((2*pi/3+pi/4*(\y-2)) r)}, {\r*sin((2*pi/3+pi/4*(\y-2)) r)}) 
plot[only marks, mark=*, mark options={fill=white},mark size=\k] 
coordinates{({\r*cos((2*pi/3+pi/4*(\y-2)) r)}, {\r*sin((2*pi/3+pi/4*(\y-2)) r)})};}
\foreach \y/\xtext in {1,2,3}{
\draw[shift={({\r*cos(4*pi/3 r)}, {\r*sin(4*pi/3 r)})}] (0pt,0pt) 
-- ++({\r*cos((4*pi/3+pi/4*(\y-2)) r)}, {\r*sin((4*pi/3+pi/4*(\y-2)) r)})
plot[only marks, mark=*, mark options={fill=white},mark size=\k] 
coordinates{({\r*cos((4*pi/3+pi/4*(\y-2)) r)}, {\r*sin((4*pi/3+pi/4*(\y-2)) r)})};}
\foreach \y/\xtext in {1,2,3}{
\draw[shift={({\r*cos(0 r)}, {\r*sin(0 r)})}] (0pt,0pt) 
-- ++({\r*cos((pi/4*(\y-2)) r)}, {\r*sin((pi/4*(\y-2)) r)})
plot[only marks, mark=*, mark options={fill=white},mark size=\k] 
coordinates{({\r*cos((pi/4*(\y-2)) r)}, {\r*sin((pi/4*(\y-2)) r)})};}
\foreach \x/\xtext in {1,2,3}{
\draw plot[only marks, mark=*, mark options={fill=red},mark size=\k] 
coordinates{({\r*cos(\x*2*pi/3 r)}, {\r*sin(\x*2*pi/3 r)})};}
\end{tikzpicture}
\hspace{1cm}
\begin{tikzpicture}[scale=0.7,
bor/.style={circle, draw, only marks, mark=*, fill=white,mark size=\k},
rel/.style={rectangle,draw,rounded corners=0.6ex,  fill=white, minimum size=\k},
coo/.style={circle, draw, only marks, mark=*, fill=red, minimum size=\k+3pt},
]
\draw (-2*\r,0) node[bor]{}-- (0,0) node[coo]{}-- (0,-\r)node[bor]{}  --(0,\r)node[bor]{} --(0,2*\r)node[bor]{} --(0,0) -- (\r,0)node[bor]{} --(3*\r,0)node[bor]{};
\foreach \x/\xtext in {3,5}{
\draw[shift={(-\r,0)}] (0pt,0pt) node[coo]{} -- ({\r*cos(\x*pi/4 r)}, {\r*sin(\x*pi/4 r)})
node[bor]{};}
\foreach \x/\xtext in {-1,1}{
\draw[shift={(2*\r,0)}] (0pt,0pt) node[coo]{}-- ({\r*cos(\x*pi/4 r)}, {\r*sin(\x*pi/4 r)})
node[bor]{};
}
\end{tikzpicture}

\caption{A nonlinear and a linear tree on 13 vertices, with HDV's in red.}
\label{fig:10-NL-tree}
\end{figure}
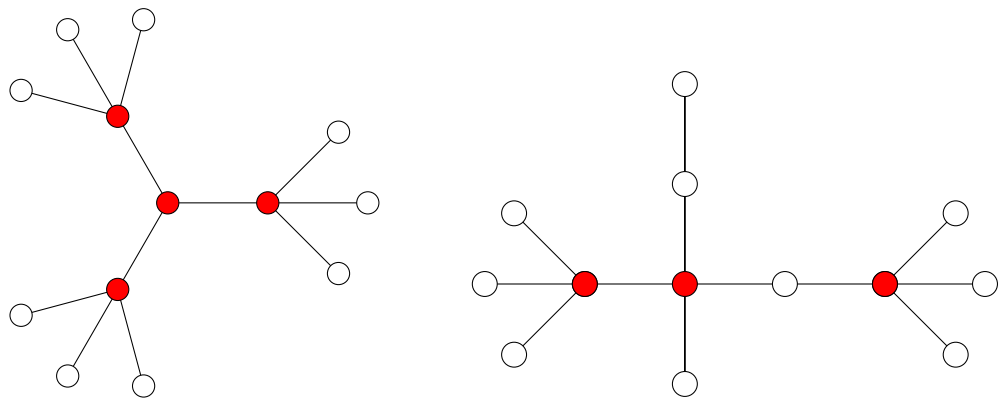
\end{center}

In \cite{TwoAndrews}, the notion of a \textit{linear tree} was introduced. A \textit{high degree vertex} (HDV) in a tree is a vertex of degree $\geq 3$, and a tree is called linear if all HDV's lie on a single induced path. A linear tree with $k$ HDV's is called $k$-linear. An example of a $3$-linear tree is given in Figure \ref{fig:10-NL-tree}. All trees on fewer than $10$ vertices are linear, but although the number of linear trees grows rapidly with $n$, eventually the fraction of trees which are linear goes to $0$ \cite{Wakhare}. However, many small trees are linear; for instance, $37.2\%$ of the $104,636,890$ nonisomorphic trees on $25$ vertices are linear. The purpose of introducing linear trees in \cite{TwoAndrews} was to study the multiplicity list problem. Since not only multiplicity lists, but also the associated inverse eigenvalue problem (IEP), were fully understood for paths and generalized stars \cite{Johnson1, Johnson3}, and any linear tree may be canonically decomposed into these, a superposition principle (generalizing that of \cite{Johnson1}) was introduced to generate the spectra of linear trees. This Linear Superposition Principle (LSP) was shown to produce all multiplicity lists that could possibly occur for a linear tree. It was also shown, using the implicit function theorem (IFT), that for certain subclasses of linear trees, all lists that the LSP produces do occur \cite{TwoAndrews}. This left the important question of whether the LSP produces exactly $\mathcal{L}_0(T)$ for every linear tree $T$. Here, we answer that question affirmatively, settling a major portion of the multiplicity list problem. This is based upon an innovative way to apply the implicit function theorem, using some important facts about the symmetric inverse eigenvalue problem for paths. A corollary of our main result is that, for linear trees, the multiplicity list problem and IEP are equivalent. This is not always the case for nonlinear trees \cite{Johnson3}. 

In Section \ref{sec:background} we provide some necessary background. Section \ref{sec3} contains our main result: every multiplicity list generated by the LSP can be achieved by a linear tree. In Section \ref{sec4} we explore several consequences of our main result, and in Section \ref{sec5} we prove that all linear trees have diameter equal to their minimum number of distinct eigenvalues.

\section{Background}\label{sec:background}
A \textit{high degree vertex} (HDV) in a tree is a vertex of degree $\geq 3$, and a tree is called \textit{linear} if all HDV's lie on a single induced path \cite{TwoAndrews}. Such a path is called a \textit{central path}, and we assume that a particular one has been identified. A \textit{generalized star} is a tree with at most one HDV; equivalently, it has a central vertex (which is arbitrary in the case of paths), to which arbitrarily long arms are appended. An arm is a path of at least one vertex, not including the central vertex. Each of the HDV's in a linear tree is the center vertex of an induced generalized star; assume there are $k$ of them. Since they lie on the central path, they are connected by paths $s_1,s_2,\ldots,s_{k-1}$, some of them possibly ``empty" (i.e., a single edge, whose vertices are consecutive generalized star centers). So, we may view a linear tree as $T_1,s_1,T_2,s_2,\ldots, s_{k-1},T_k$, in which $s_i$ is a path of length $|s_i|$ and $T_i$ is a generalized star with a center on the central path and arms of length $\ell_{i1},\dots,\ell_{ih}$, measured in vertices. For example, from left to right, $|s_1| = 0$ and $|s_2|=1$ in the linear tree of Figure \ref{fig:10-NL-tree}. For notational convenience, we will occasionally drop the absolute value if the context is unambiguous. If a star is a path, we take its center to be an interior vertex.

We require several concepts from the theory of partitions. A \textit{partition} of a positive integer $n$ is a decomposition of $n$ into a sum of positive integers $\eta_1,\ldots,\eta_k$ such that $\eta_1 \geq \cdots \geq \eta_k$ and $\sum_{i=1}^k \eta_i = n$. Given a partition, recall its \textit{Ferrers diagram} is a left justified array of boxes in which the $i$-th row has $\eta_i$ boxes. Transposition of a Ferrers diagram gives the {dual partition}, denoted by $(\eta_1,\ldots,\eta_k)^*$. Note that the number of parts in a partition and its dual may differ. There is also natural partial ordering on the set of partitions called \textit{majorization}. The partition $\eta=(\eta_1,\ldots,\eta_k)$ is majorized by the partition $\nu =(\nu_1,\ldots,\nu_r)$,  denoted by $\eta \preceq \nu$, if they satisfy the inequalities $\eta_1\leq \nu_1, \eta_1+\eta_2 \leq \nu_1+\nu_2, \ldots, \sum_{i=1}^k \eta_i \leq \sum_{i=1}^r \nu_i$. We require the last inequality to be an equality, and interpret $\eta_i,\nu_i = 0$ if the index exceeds the length of the respective vector.

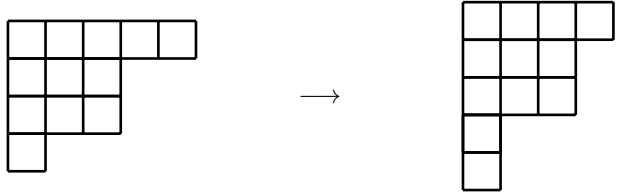
\begin{figure}
\begin{center}
\begin{minipage}{50mm}
\begin{center}
\begin{tikzpicture}[scale=0.5, line width=1pt]
  \draw (0,0) grid (5,1);
  \draw (0,0) grid (3,-1);
  \draw (0,-1) grid (3,-2);
  \draw (0,-2) grid (1,-3);
\end{tikzpicture}
\end{center}
\end{minipage}
$\longrightarrow$
\begin{minipage}{50mm}
\begin{center}
\begin{tikzpicture}[scale=0.5, line width=1pt]
  \draw (0,0) grid (4,1);
  \draw (0,0) grid (3,-1);
  \draw (0,-1) grid (3,-2);
  \draw (0,-2) grid (1,-3);
  \draw (0,-2) grid (1,-4);
\end{tikzpicture}
\end{center}
\end{minipage}
\end{center}
\caption{Ferrers diagrams of $12 = 5+3+3+1$ and its dual $12 = 4 + 3 + 3+1+1$ }
\end{figure}

Let $A(v)$ denote the principal submatrix of a square matrix $A$ obtained by removing the row and column corresponding to index/vertex $v$, $A[T_i]$ the principal submatrix of $A$ corresponding to vertices lying in the subtree $T_i$, $\sigma(A)$ denote the spectrum of $A$, and $m_A(\lambda)$ the multiplicity of the eigenvalue $\lambda$ in the matrix $A$. An eigenvalue is \textit{upward} relative to index/vector $v$ and matrix $A$ if $m_{A(v)}(\lambda) = m_A(\lambda)+1$.

We use $\hat{\mathcal{L} }_0(T)$ to denote the multiplicity lists in which upward multiplicities for $v$ are distinguished with hats. This results in a multiplicity list mixing upward and non-upward eigenvalues. Note that eigenvalues with upward multiplicity $\hat{0}$ are allowed. For a generalized star, it is implicit that $\hat{\mathcal{L} }_0(T)$ will always denote upward multiplicities with respect to the removal of the central vertex.


\begin{thm} \cite{Johnson1}\label{gstarmult}
Let $T$ be a generalized star on $n$ vertices with central vertex $v$ of degree $k$ and arm lengths $l_1\geq \dots \geq l_k$. Then $\hat{q} =  {(q_1,\cdots,q_r)} \in \hat{\mathcal{L}}_0(T)$ if and only if $\hat{q}$ satisfies the following conditions:
\begin{enumerate}
\item $q_i$ is a nonnegative integer, $1\leq i\leq r,$ and $\sum_{i=1}^r q_i =n$;
\item if $q_i$ is an upward multiplicity in $\hat{q}$, then $1<i<r$ and neither $q_{i-1}$ nor $q_{i+1}$ is an upward multiplicity in $\hat{q}$;
\item $(q_{i_1}+1,\ldots,q_{i_h}+1)_e \preceq (l_1,\ldots,l_k)^*$, in which $q_{i_1}\geq \cdots \geq q_{i_h}$ are the upward multiplicities of $\hat{q}$ greater than $1$, and $(q_{i_1}+1,\ldots,q_{i_h}+1)_e$ means that the vector is augmented with $e$ ones so that $e+\sum_{m=1}^h (q_{i_m}+1) = \sum_{m=1}^k l_m.  $
\end{enumerate} 
\end{thm}

As a corollary, $\hat{\mathcal{L}}_0(T)$ for a generalized star alternates upward and non-upward eigenvalues \cite[Theorem 9]{TwoAndrews}. Given this fact, which completely characterizes the allowable eigenvalues for a matrix whose graph is a generalized star, \cite{TwoAndrews} proposed the Linear Superposition Principle (LSP) for the possible multiplicity lists in a real symmetric matrix whose graph is a linear tree.

\begin{definition}[LSP] \label{lspstatement}
Let $T_1,\ldots,T_k$ be generalized stars and $s_1,\ldots, s_{k-1}$ nonnegative integers. Given $\hat{b}_i$ an upward multiplicity list for $T_i$ (with respect to the central vertex), $i=1,\ldots,k$, and $\hat{c}_j$ a list of $s_j$ non-upward ones, $j=1,\ldots,k-1$, construct augmented lists $b_i^{+}, i =1,\ldots,k$, and $c_j^{+}, j=1,\ldots,k-1,$ subject to the following conditions:
\begin{enumerate}
\item all $b_i^{+}$'s and $c_j^{+}$'s are the same length;
\item each $b_i^{+}$ and $c_j^{+}$ is obtained from its corresponding $\hat{b}_i$ and $\hat{c}_j$ by inserting non-upward $0$'s;
\item for each $l$, the $l$-th element of the augmented lists, denoted $b_{i,l}^{+}$ and $c_{j,l}^{+}$, are not all non-upward $0$'s;
\item for each $l$, arranging the $b_{i,l}^{+}$'s and $c_{j,l}^{+}$'s in the order $b_{1,l}^{+},c_{1,l}^{+},b_{2,l}^{+}c_{2,l}^{+},\ldots,b_{k,l}^{+}$, there is at least one upward multiplicity between any two non-upward ones.
\end{enumerate}
Then $\sum_{i=1}^k b_i^{+}+\sum_{j=1}^{k-1} c_j^{+}$, where the addition is termwise, is a multiplicity list for $LT(T_1,s_1,\ldots,s_{k-1},T_k)$ generated by the LSP.
\end{definition}
The LSP may be represented in tabular form, as in Table \ref{lps_table}. The LSP is then equivalent to completing the given table so that 
\begin{enumerate}
\item $b_i^+$ is the multiplicity list $\hat{b}_i$ along with some added non-upward zeros;
\item $c_i^+$ contains $s_i$ non-upward ones and the remaining entries are non-upward zeros;
\item no column has all non-upward zeros;
\item if a column contains two non-upward ones, they are separated by an element with upward multiplicity.
\end{enumerate}

\begin{table}
\centering
\begin{tabular}{llllll}
                            & $\lambda_1$             & $\lambda_2$             & $\ldots$ & $\ldots$ & $\lambda_\nu$ \\ \cline{2-6}
\multicolumn{1}{l|}{$b_1^+$}     & \multicolumn{1}{l|}{} & \multicolumn{1}{l|}{} & \multicolumn{1}{l|}{} & \multicolumn{1}{l|}{} & \multicolumn{1}{l|}{} \\ \cline{2-6} 
\multicolumn{1}{l|}{$c_1^+$}     & \multicolumn{1}{l|}{} & \multicolumn{1}{l|}{} & \multicolumn{1}{l|}{} & \multicolumn{1}{l|}{} & \multicolumn{1}{l|}{} \\ \cline{2-6} 
\multicolumn{1}{l|}{$b_2^+$}     & \multicolumn{1}{l|}{} & \multicolumn{1}{l|}{} & \multicolumn{1}{l|}{} & \multicolumn{1}{l|}{} & \multicolumn{1}{l|}{} \\ \cline{2-6} 
\multicolumn{1}{l|}{$c_2^+$}     & \multicolumn{1}{l|}{} & \multicolumn{1}{l|}{} & \multicolumn{1}{l|}{} & \multicolumn{1}{l|}{} & \multicolumn{1}{l|}{} \\ \cline{2-6} 
\multicolumn{1}{c|}{$\vdots$}    & \multicolumn{1}{l|}{} & \multicolumn{1}{l|}{} & \multicolumn{1}{l|}{} & \multicolumn{1}{l|}{} & \multicolumn{1}{l|}{} \\ \cline{2-6} 
\multicolumn{1}{l|}{$c_{k-1}^+$} & \multicolumn{1}{l|}{} & \multicolumn{1}{l|}{} & \multicolumn{1}{l|}{} & \multicolumn{1}{l|}{} & \multicolumn{1}{l|}{} \\ \cline{2-6} 
\multicolumn{1}{l|}{$b_k^+$}     & \multicolumn{1}{l|}{} & \multicolumn{1}{l|}{} & \multicolumn{1}{l|}{} & \multicolumn{1}{l|}{} & \multicolumn{1}{l|}{} \\ \cline{2-6} 
sum                            & $a_1$                 & $a_2$                 & $\ldots$              & $\ldots$              & $a_\nu$              
\end{tabular}
\vspace{5mm}
\caption{The tabular form of the LSP}
\label{lps_table}
\end{table}

We will also use the following theorems from multiplicity theory in the proof of Lemma \ref{lem10} below.
\begin{thm}[Parter--Weiner, etc.]\cite{JohnsonPW} \label{pw} 
Let $T$ be a tree and $A$ a matrix in $\mathcal{S}(T)$. Suppose that there is a vertex $v$ of $T$ and a real number $\lambda$ such that $\lambda \in \sigma(A) \cap \sigma(A(v))$. Then
\begin{enumerate}
\item there is a vertex $u$ of $T$ such that $m_{A(u)}(\lambda) = m_A(\lambda)+1$;
\item if $m_A(\lambda)\geq 2$, then $\lambda \in \sigma(A) \cap \sigma(A(v))$ is automatically satisfied and $u$ may be chosen so that $\deg_T(u)\geq 3$ and so that there are at least three components $T_1$, $T_2$, and $T_3$ of $T \setminus u$ such that $m_{A[T_i]}(\lambda)\geq 1$, $i=1,2,3$; and
\item if $m_A(\lambda)\geq 1$, then $u$ may be chosen so that $\deg_T(u)\geq 2$ and so that there are two components $T_1$ and $T_2$ of $T \setminus u$ such that $m_{A[T_i]}(\lambda)\geq 1$, $i=1,2$.
\end{enumerate}
\end{thm}

\begin{thm}[Interlacing inequalities] \cite[Chapter 4]{Horn}\label{interlacing}
Let $A\in M_n(\mathbb{C})$ be a Hermitian matrix with eigenvalues $$\alpha_1 \leq \cdots \leq \alpha_n.$$ Let $B\in M_{n-1}(\mathbb{C})$ be the principal submatrix resulting from the deletion of row and column $i$ of $A$. Suppose the Hermitian matrix $B$ has eigenvalues 
$$\beta_1 \leq \cdots \leq \beta_{n-1}. $$ 
Then we have the inequalities 
$$\alpha_1 \leq \beta_1 \leq \alpha_2 \leq\cdots \leq \beta_{n-1} \leq \alpha_n. $$
\end{thm}
The interlacing inequalities significantly constrain the multiplicity lists of a real symmetric or Hermitian matrix and its principal submatrices. For instance, if $A$ has eigenvalue $\lambda$ with multiplicity $m_A(\lambda)$ and we remove the row and column corresponding to a vertex $v$ in the graph of $A$, we have
$$ |m_{A(v)}(\lambda) -m_A(\lambda)| \leq 1.$$

We also use the \textit{neighbors formula}, a combinatorial expansion of the determinant with references to the underlying graph of a matrix.
\begin{thm}[Neighbors Formula, surveyed in \cite{Johnson3}] \label{neighbors}
 When the underlying graph of a Hermitian matrix $A=(a_{ij})$ is a tree $T$, consider a particular vertex $v$ with neighbors $u_1,\ldots,u_k$. Let $T_j$ be the branch of $T$ at $v$ which contains $u_j$. Let $p_A(t)=\det(tI -A )$ denote the characteristic polynomial of $A$. Then 
$$ p_A(t) = (t-a_{vv}) \prod_{j=1}^k p_{A[T_j]}(t) -\sum_{j=1}^k |a_{vu_j}|^2p_{A[T_j-u_j]}(t)\prod_{ \substack{\ell=1 \\ \ell \neq j }  }^k p_{A[T_\ell]}(t).$$
\end{thm}

\subsection{Example of the LSP} The following figure \cite[Example 21]{TwoAndrews} gives a $3$-linear tree $T:=LT(T_1,s_1,T_2,s_2,T_3)$ which has been labelled with its canonical decomposition into paths of length $s_1=1, s_2=1$ and generalized stars $T_1,T_2,T_3$.
\begin{center}
\begin{figure}[!ht]
\centering
\def\r{1.9}
\def\k{12pt}
\begin{tikzpicture}[scale=1.0,
bor/.style={circle, draw, only marks, mark=*, fill=white,mark size=\k+3},
rel/.style={rectangle,draw,rounded corners=0.6ex,  fill=white, minimum size=\k},
coo/.style={circle, draw, only marks, mark=*, fill=white, minimum size=\k+3},
]
\draw (0,0)--({\r*cos(3*pi/4 r)}, {\r*sin(3*pi/4 r)})node[coo]{$T_2$};
\draw (0,0)--({\r*cos(pi/4 r)}, {\r*sin(pi/4 r)})node[coo]{$T_2$};
\draw (-2*\r,0) node[coo]{}-- (-\r,0) node[coo]{}-- (0,0) node[coo]{$T_2$}-- (0,-\r)node[coo]{$T_2$}   --(0,0) -- (\r,0)node[coo]{} --(2*\r,0);
\foreach \x/\xtext in {3,5}{
\draw[shift={(-2*\r,0)}] (0pt,0pt) node[coo]{$T_1$} -- ({\r*cos(\x*pi/4 r)}, {\r*sin(\x*pi/4 r)})
node[coo]{$T_1$};}
\foreach \x/\xtext in {-1,1}{
\draw[shift={(2*\r,0)}] (0pt,0pt) node[coo]{$T_3$}-- ({\r*cos(\x*pi/4 r)}, {\r*sin(\x*pi/4 r)})
node[coo]{$T_3$};
}
\end{tikzpicture}
\caption{A $3$-linear tree on $12$ vertices.}
\end{figure}
\end{center}

The allowable upward multiplicity lists for each of the trees are as follows, where each upward multiplicity is
\begin{align*}
\hat{\mathcal{L}}_0(T_1) &= \{  (1,\hat{1},1),(1,\hat{0},1,\hat{0},1) \}, \\
\hat{\mathcal{L}}_0(T_2) &= \{ (1,\hat{2},1),(1,\hat{1},1,\hat{0},1),(1,\hat{0},1,\hat{1},1),(1,\hat{0},1,\hat{0},1,\hat{0},1)  \}, \\
\hat{\mathcal{L}}_0(T_3) &= \{(1,\hat{1},1),(1,\hat{0},1,\hat{0},1)\}.
\end{align*}
These exhaust all of the multiplicity lists compatible with Theorem \ref{gstarmult}. The LSP will then generate the following unordered multiplicity lists:
\begin{align*}
\{(6,1,1,1,1,1,1),&(5,2,1,1,1,1,1),(5,1,1,1,1,1,1,1),(4,3,1,1,1,1,1),(4,2,2,1,1,1,1),(4,2,1,1,1,1,1,1), \\
&(4,1,1,1,1,1,1,1,1),(3,3,2,1,1,1,1),(3,3,1,1,1,1,1,1),(3,2,2,1,1,1,1,1),\\
&(3,2,1,1,1,1,1,1,1),(3,1,1,1,1,1,1,1,1,1),(2,2,2,2,2,1,1),(2,2,2,2,1,1,1,1),\\
&(2,2,2,1,1,1,1,1,1),(2,2,1,1,1,1,1,1,1,1),(2,1,1,1,1,1,1,1,1,1,1)\}.
\end{align*}
The main result of \cite{TwoAndrews} said that the set of possible unordered multiplicity lists ${\mathcal{L}}(T)$ was a subset of these lists generated by the LSP. Our main result is that they are in fact equal. Two valid superpositions generated by the LSP leading to multiplicity lists are given in Table \ref{lspexample}.

\begin{table}
\centering
{\renewcommand{\arraystretch}{1.5}
\begin{tabular}{llllllll}
                            & $\lambda_1$             & $\lambda_2$             & $\lambda_3$ & $\lambda_4$ & $\lambda_5$& $\lambda_6$& $\lambda_7$ \\ \cline{2-8}
\multicolumn{1}{l|}{$b_1^+$}     & \multicolumn{1}{l|}{$1$} & \multicolumn{1}{l|}{$0$} & \multicolumn{1}{l|}{$0$} & \multicolumn{1}{l|}{ $\widehat{1}$ } & \multicolumn{1}{l|}{$0$}& \multicolumn{1}{l|}{$0$}& \multicolumn{1}{l|}{$1$} \\ \cline{2-8} 
\multicolumn{1}{l|}{$c_1^+$}     & \multicolumn{1}{l|}{$0$} & \multicolumn{1}{l|}{$0$} & \multicolumn{1}{l|}{$0$} & \multicolumn{1}{l|}{$1$} & \multicolumn{1}{l|}{$0$} & \multicolumn{1}{l|}{$0$}& \multicolumn{1}{l|}{$0$}\\ \cline{2-8} 
\multicolumn{1}{l|}{$b_2^+$}     & \multicolumn{1}{l|}{$0$} & \multicolumn{1}{l|}{$1$} & \multicolumn{1}{l|}{$0$} & \multicolumn{1}{l|}{$\widehat{2}$} & \multicolumn{1}{l|}{$0$}& \multicolumn{1}{l|}{$1$}& \multicolumn{1}{l|}{$0$} \\ \cline{2-8} 
\multicolumn{1}{l|}{$c_2^+$}     & \multicolumn{1}{l|}{$0$} & \multicolumn{1}{l|}{$0$} & \multicolumn{1}{l|}{$0$} & \multicolumn{1}{l|}{$1$} & \multicolumn{1}{l|}{$0$}& \multicolumn{1}{l|}{$0$}& \multicolumn{1}{l|}{$0$} \\ \cline{2-8} 
\multicolumn{1}{l|}{$b_3^+$}     & \multicolumn{1}{l|}{$0$} & \multicolumn{1}{l|}{$0$} & \multicolumn{1}{l|}{$1$} & \multicolumn{1}{l|}{$\widehat{1}$} & \multicolumn{1}{l|}{$1$}& \multicolumn{1}{l|}{$0$}& \multicolumn{1}{l|}{$0$} \\ \cline{2-8} 
sum                            & $1$                 & $1$                 & $1$              & $6$              & $1$       & $1$       & $1$              
\end{tabular}}
\hspace{5mm}
{\renewcommand{\arraystretch}{1.5}
\begin{tabular}{lllllllll}
                            & $\lambda_1$             & $\lambda_2$             & $\lambda_3$ & $\lambda_4$ & $\lambda_5$& $\lambda_6$& $\lambda_7$& $\lambda_8$ \\ \cline{2-9}
\multicolumn{1}{l|}{$b_1^+$}     & \multicolumn{1}{l|}{$0$} & \multicolumn{1}{l|}{$1$} & \multicolumn{1}{l|}{$0$} & \multicolumn{1}{l|}{ $\widehat{1}$ } & \multicolumn{1}{l|}{$0$}& \multicolumn{1}{l|}{$1$}& \multicolumn{1}{l|}{$0$}& \multicolumn{1}{l|}{$0$} \\ \cline{2-9} 
\multicolumn{1}{l|}{$c_1^+$}     & \multicolumn{1}{l|}{$0$} & \multicolumn{1}{l|}{$0$} & \multicolumn{1}{l|}{$0$} & \multicolumn{1}{l|}{$1$} & \multicolumn{1}{l|}{$0$} & \multicolumn{1}{l|}{$0$}& \multicolumn{1}{l|}{$0$}& \multicolumn{1}{l|}{$0$}\\ \cline{2-9} 
\multicolumn{1}{l|}{$b_2^+$}     & \multicolumn{1}{l|}{$0$} & \multicolumn{1}{l|}{$0$} & \multicolumn{1}{l|}{$1$} & \multicolumn{1}{l|}{$\widehat{0}$} & \multicolumn{1}{l|}{$1$}& \multicolumn{1}{l|}{$\widehat{1}$}& \multicolumn{1}{l|}{$1$} & \multicolumn{1}{l|}{$0$}\\ \cline{2-9} 
\multicolumn{1}{l|}{$c_2^+$}     & \multicolumn{1}{l|}{$0$} & \multicolumn{1}{l|}{$0$} & \multicolumn{1}{l|}{$0$} & \multicolumn{1}{l|}{$1$} & \multicolumn{1}{l|}{$0$}& \multicolumn{1}{l|}{$0$}& \multicolumn{1}{l|}{$0$} & \multicolumn{1}{l|}{$0$}\\ \cline{2-9} 
\multicolumn{1}{l|}{$b_3^+$}     & \multicolumn{1}{l|}{$1$} & \multicolumn{1}{l|}{$0$} & \multicolumn{1}{l|}{$0$} & \multicolumn{1}{l|}{$0$} & \multicolumn{1}{l|}{$0$}& \multicolumn{1}{l|}{$\widehat{1}$}& \multicolumn{1}{l|}{$0$}& \multicolumn{1}{l|}{$1$} \\ \cline{2-9} 
sum                            & $1$                 & $1$                 & $1$              & $3$              & $1$       & $3$       & $1$   & $1$              
\end{tabular}}
\vspace{5mm}
\caption{Two example superpositions from the LSP}
\label{lspexample}
\end{table}

\section{Sufficiency of the LSP}\label{sec3}

In \cite{TwoAndrews}, the authors obtained a necessary condition for a $k$-linear tree to have ordered multiplicity list $L=(a_1,\ldots, a_\nu)$ from its constituent generalized star and path lists. {Here, our purpose is to show the sufficiency of these conditions, so that any multiplicity list obtained from the LSP is achieved by some matrix for the given linear tree}. This was previously shown for all linear trees with two HDV's and all depth one linear trees (caterpillars) \cite{TwoAndrews}, and conjectured there for all linear trees. The full converse has a number of interesting consequences which will be discussed later.


Our basic method is similar to that in \cite{TwoAndrews}, i.e., to use the Implicit Function Theorem (IFT). However, instead of \cite[Lemma 18]{TwoAndrews}, we use Lemma \ref{lem10}, below.
\begin{thm}[Implicit Function Theorem]
Let $f: \mathbb{R}^{n+m} \to \mathbb{R}^n$ be a continuously differentiable function. Suppose that, for $x_0\in \mathbb{R}^n$ and $y_0\in \mathbb{R}^m$, $f(x_0,y_0)=0$ and the Jacobian $\frac{\partial f}{\partial x}(x_0,y_0)$ is invertible. Then there exists a neighborhood $U\subset \mathbb{R}^m$ around $y_0$ such that $f(x,y)=0$ has a unique solution $x$ for any fixed $y\in U$. Furthermore,  the solution $x$ as a function of $y$ is continuous at $y_0$.
\end{thm}
The essence of the Jacobian method is to view a matrix with graph $G$ as a multivariable function of the entries. We then find a subgraph $G_0 \subset G$ such that $A^{(0)}$ with graph $G_0$ has a spectrum with multiplicities we desire. We will always consider $G_0$ with the same number of vertices at $G$. So long as we can select enough variables, henceforth called \textit{implicit entries}, to find a nonsingular Jacobian with respect to those entries, we can then perturb the edges in $G \setminus G_0$ by some sufficiently small $\epsilon$. By letting $f$ be a set of eigenvalue conditions ensuring that the resulting graph has the spectrum we desire, the resulting matrix will have graph $G$ but the same spectrum as $A^{(0)}$.

We now appeal to the following lemma, which reduces the difficulty of checking the nonsingularity of the Jacobian when the underlying graph of $A$ is a tree.
\begin{lem}\cite{Johnson4}\label{lem8}
Let T be a tree of order $n$ and $F=(f_k)$, $f_k(A) = \det (A[S_k]-\lambda_kI)$ a vector of $r\leq n$ determinant conditions, with $S_k \subseteq \{1,2,\ldots,n\}$ a given index subset and $\lambda_k$ a real number, $k=1,\ldots,r$. Assume $r$ implicit entries have been identified. Suppose that a real symmetric matrix $A^{(0)}$, whose graph is a subgraph of $T$, is the direct sum of irreducible matrices $A^{(0)}_1,\ldots,A^{(0)}_p$. Let $J(A^{(0)})$ be the Jacobian matrix of $F$ with respect to the implicit entries evaluated at $A^{(0)},$ and suppose 
\begin{enumerate}
\item every off-diagonal implicit entry in $A^{(0)}$ has a nonzero value;
\item for each $k=1,\ldots,r$, $f_k(A^{(0)}_l)=0$ for exactly one $l\in \{1,2,\ldots,p\}$;
\item for each $l=1,\ldots,p$, the columns of $J(A^{(0)})$ associated with the implicit entries of $A^{(0)}_l$ are linearly independent.
\end{enumerate} 
Then $J(A^{(0)})$ is nonsingular.
\end{lem}
The Jacobian reduces to a block diagonal form, so that we can merely check the nonsingularity of each component's Jacobian instead. We therefore only need to examine the case of a single generalized star. We can use the previous lemma to ``chain together'' multiple generalized stars into a linear tree, hence proving the sufficiency of the LSP conditions. Lemma \ref{lem10}, though extremely technical, ensures that we can select enough implicit entries to create a nonsingular Jacobian.  

Importantly, we also use a correspondence between the entries of a tridiagonal matrix and the eigenvalues of it, and its leading $(n-1)$-by-$(n-1)	$ principal submatrix.
\begin{thm}\cite{Gray}\label{JacSpectral} 
Given $\{\omega_1,\ldots,\omega_n\}$ and $\{\mu_1,\ldots,\mu_{n-1}\}$ which interlace as follows:
$$\omega_1< \mu_1<\omega_2<  \cdots < \mu_{n-1} < \omega_n,$$
there exists a unique symmetric tridiagonal matrix $J$ with positive off-diagonal entries, such that the spectrum of $J$ is $\{\omega_k\}_{k=1}^n$ and the spectrum of $J$ with its last row and column removed is $\{\mu_k\}_{k=1}^{n-1}$.
\end{thm}
Therefore, instead of regarding such a matrix as determined by its $2n-1$ matrix entries, we can regard it as determined (up to signature similarity) by $2n-1$ eigenvalues. A signature matrix is a matrix with $\pm 1$ entries on the diagonal, and signature similarity denotes conjugation by a signature matrix, which flips the signs of off-diagonal entries in a structured manner. Given a symmetric tridiagonal matrix, we can select a signature matrix $S$ so that after conjugation by $S$ every off-diagonal entry is positive.

The following lemma is a direct generalization of \cite[Lemma 18]{TwoAndrews}, which applied to simple stars (each arm has length $1$). We refer the reader to the proof of \cite[Lemma 18]{TwoAndrews} for further discussions of the matrix theoretic results we use. Note that the function $F(A)$ is a $u+1$ dimensional vector, which enforces $u+1$ eigenvalue constraints when set equal to $0$. Matrix $A$ can be regarded as a multivariable function of the matrix entries. Our goal is to perturb $u+1$ of these matrix entries of the matrix, considered as a multivariate function, while maintaining nonsingularity of the Jacobian. We first compose the $F(A)$ function with a continuous change of variables which maps several subsets of the matrix entries to related eigenvalues $\eta_k^{(i)}$, which are a nonlinear function of matrix entries. We then select a total of $u+1$ matrix entries and eigenvalues $\eta_k^{(i)}$ according to the procedure detailed at the beginning of Lemma \ref{lem10}. In terms of mixed matrix entry and eigenvalue coordinates, we obtain a reasonably well-behaved Jacobian, from which we can deduce nonsingularity.

We are now dealing with four sets of eigenvalues. The non-upwards $\lambda$ eigenvalues give eigenvalue constraints, which we form the Jacobian with respect to. The perturbation we perform still preserves $F(A)=0$, meaning that we obtain the desired multiplicity list. The upwards $\hat{\mu}$ eigenvalues can be assigned to each vertex (see \cite{Johnson3} for background about this assignment procedure), in a way which preserves them under perturbation. Finally, there is a set of interlacing $\gamma$ and $\eta$ eigenvalues associated to each arm of our generalized star.



Through the proof of Lemma \ref{lem10}, we will use the generalized star $\mathcal{T}_7$ from Figure \ref{fig2} as a running example. This is a generalized star on $7$ vertices with ordered upward multiplicity list $\hat{\mathcal{L}}_0(\mathcal{T}_7) = (1,\hat{2},1,\hat{0},1,\hat{0},1,\hat{0},1),$ corresponding to any numerical values of the eigenvalues $(\lambda_1,\hat{\mu}_2,\lambda_3,\hat{\mu}_4,\lambda_5,\hat{\mu}_6,\lambda_7,\hat{\mu}_8,\lambda_9)$. This will be a running example through the proof of Lemma \ref{lem10}, as we construct a $5\times 5$ nonsingular Jacobian for a matrix with this graph. A single unique instance of each $\hat{\mu}_{2i}$ has been outlined blue, and dictates what variables we construct a Jacobian with respect to. There are many possible choices, but any of them leads to a nonsingular Jacobian.

\begin{center}
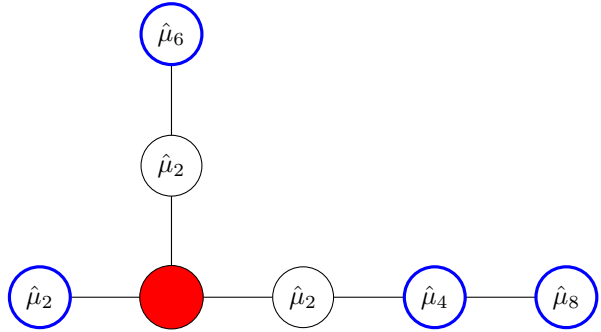
\begin{figure}[!ht]
\centering
\def\r{2.5}
\def\k{6pt}
\begin{tikzpicture}[scale=0.7,
bor/.style={circle, draw, only marks, mark=*, fill=white,mark size=\k},
blue/.style={circle, draw=blue, very thick, only marks, mark=*, fill=white,mark size=\k},
coo/.style={circle, draw, only marks, mark=*, fill=red, minimum size=\k+18pt},
]
\draw (-1*\r,0) node[blue]{$\hat{\mu}_2$}-- (0,0) node[coo]{}-- (0,\r)node[bor]{$\hat{\mu}_2$}  --(0,2*\r)node[blue]{$\hat{\mu}_6$}  --(0,0) -- (\r,0)node[bor]{$\hat{\mu}_2$} --(2*\r,0)node[blue]{$\hat{\mu}_4$}--(3*\r,0)node[blue]{$\hat{\mu}_8$};
\end{tikzpicture}
\caption{Assignment to a $7$ vertex generalized star}
\label{fig2}
\end{figure}
\end{center}

The statement of the following Lemma is identical to \cite[Lemma 18]{TwoAndrews}, but applies to generalized stars instead of simple stars.
\begin{lem}\label{lem10}
Let $T$ be a generalized star on $n$ vertices, and suppose we have a matrix valued function $A(a_1,a_2,\ldots,a_n)$ defined on $n$ real variables, whose range is the set of real symmetric matrices with graph $T$ (or some subgraph of $T$),  with arm lengths $\ell_1,\ldots, \ell_k$. Let upward multiplicity list $\hat{q} = ( {q_1,\dots,q_{2u+1}})$ have distinct upward eigenvalues $\hat{\mu}_2,\ldots,\hat{\mu}_{2u}$ and non-upward eigenvalues $\lambda_1, \lambda_3,\ldots,\lambda_{2u+1}$. Consider a matrix $A^{(0)}$ in the range of $A$ with graph $T$ and non-upwards eigenvalues $\lambda_1,\lambda_3,...,\lambda_u$. We may select ${u+1}$ variables of $A$ so that the Jacobian of the function $F(A)=(\det(A - \lambda_1I),\det(A- \lambda_3I),...,\det(A-\lambda_{2u+1}I))$ is nonsingular at $A^{(0)}$.
\end{lem}


\begin{proof}
Since $A\in \mathcal{S}(T)$ is a generalized adjacency matrix associated to the generalized star $T$, it is a square matrix of size $\left(1+\sum_{i=1}^k\ell_i \right)$ by $\left(1+\sum_{i=1}^k\ell_i\right)$. First note that the upward multiplicity list alternates upward and non-upward eigenvalues (including upward zeros, if necessary), so that the $\lambda_i$ and $\mu_i$ interlace. Therefore the hypothesis that upward $\hat{\mu}_i$ and non-upward $\lambda_i$ alternate is satisfied. This allows us to pair each non-upward $\lambda_i$, except for $\lambda_{2u+1}$, with an upward counterpart. We begin by assigning the upward eigenvalues $\hat{\mu}_i$ to a vertex on $\hat{q_i}+1$ distinct arms of our generalized stars $T$. By the Parter--Weiner theorem there must be a vertex $v$ such that removing $v$ increases the multiplicity of $v$; since we may take the vertex to have degree $\geq 3$, this must refer to the central vertex of our generalized star. Hence we have $m_{A(v)}(\hat{\mu}_i) = m_A(\hat{\mu}_i) +1 = \hat{q_i} +1$, and $A$ has eigenvalue $\hat{\mu}_i$ with multiplicity $\hat{q}_i$.

As we perturb our matrix, these eigenvalues are still assigned to arms, hence they must still be eigenvalues of the whole generalized star. Therefore, we only need a determinant condition on the $u+1$ non-upward eigenvalues. For background on this assignment procedure, refer to \cite{Johnson3}.

When dealing with simple stars, we can assume our matrix is of the form 
$$
\begin{bmatrix}
    a_{0} & a_{1} & a_{2} & \dots  & a_{k} \\
    a_{1} & b_1 & 0 & \dots  & 0 \\
    \vdots & \vdots & \vdots & \ddots & \vdots \\
    a_{k} & 0 & 0 & \dots  & b_k
\end{bmatrix},
$$
whose determinant can be explicitly calculated \cite{TwoAndrews}. However, since our matrix $A$ has a generalized star as its underlying graph, we can assume it is of the form
$$
A = \begin{bmatrix}
    a_{0} & a_{1} & a_{2} & \dots  & a_{k} \\
    a_{1}^T & B_1 & 0 & \dots  & 0 \\
    \vdots & \vdots & \vdots & \ddots & \vdots \\
    a_{k}^T & 0 & 0 & \dots  & B_k
\end{bmatrix},
$$
where each $a_i$ is a row vector with a single nonzero entry in the first position. Each $B_i$ is also a symmetric tridiagonal matrix, corresponding to a single arm of our star. 

We also fix some notation. If $B_i$ is the block matrix corresponding to arm $i$, we say that it has size $n_i\geq 1$ if the underlying path consists of $n_i$ vertices. Let $\{\gamma^{(i)}_1,\gamma^{(i)}_2,\ldots,\gamma^{(i)}_{n_i}\}$ denote the $n_i$ eigenvalues of $B_i$, which must necessarily be distinct since $B_i$ is a symmetric tridiagonal matrix. Let $B_i(1)$ denote the principal submatrix of $B_i$ corresponding to deleting the first row and column, and let $\{\eta^{(i)}_1,\eta^{(i)}_2,\ldots,\eta^{(i)}_{n_i-1}\}$ denote the $n_i-1$ distinct eigenvalues of $B_i(1)$, and note that they strictly interlace the $\gamma^{(i)}$ eigenvalues by the interlacing inequalities (Theorem \ref{interlacing}). Explicitly we have
\begin{equation}\label{binterlacing}
\gamma^{(i)}_1< \eta^{(i)}_1< \gamma^{(i)}_2 < \cdots < \eta^{(i)}_{n_i-1} <\gamma^{(i)}_{n_i}. 
\end{equation}
Since the determinant is the product of all the eigenvalues, we also have the expansions
$$\det(tI-B_i) = \prod_{j=1}^{n_i}\left( t - \gamma^{(i)}_j\right), \quad \det(tI-B_i(1)) = \prod_{j=1}^{n_i-1}\left( t - \eta^{(i)}_j\right). $$

Expanding the characteristic polynomial of $A$ as a function of the matrix entries is complicated leads to a Jacobian where every entry is a sum over products, from which we were unable to prove nonsingularity. Instead, \textit{our main technical innovation is to recognize that a real symmetric tridiagonal matrix is completely specified (up to signature similarity) by its eigenvalues and the eigenvalues of its principal submatrix}. The exact statement of the result is provided by Theorem \ref{JacSpectral}. Therefore, instead of expanding the characteristic polynomial of $A$ in terms of the matrix entries, we expand it in terms of the eigenvalues of each $B_i$ and $B_i(1)$, along with the matrix entries $\{a_0,a_1,\ldots,a_k\}$. This represents a smooth change of variables from the matrix regarded as a multivariable function of the matrix entries. We begin by applying the neighbors formula, expanding around the central vertex, and rewriting in our block matrix notation:

\begin{align*}
 p_A(t) &= (t-a_{0}) \prod_{j=1}^k p_{B_j}(t) -\sum_{j=1}^k a_{j}^2p_{B_j(1)}(t)\prod_{ \substack{\ell=1 \\ \ell\neq j }  }^k p_{B_\ell}(t) \\
&=  (t-a_{0}) \prod_{j=1}^k \prod_{m=1}^{n_j}\left( t - \gamma^{(j)}_m\right)-\sum_{j=1}^k a_{j}^2  \prod_{m=1}^{n_j-1}\left( t - \eta^{(j)}_m\right)   \times   \prod_{ \substack{\ell=1 \\ \ell \neq j }  }^k   \prod_{m=1}^{n_\ell}\left( t - \gamma^{(\ell)}_m\right).
\end{align*}

We had begun with $A \in M_n(\mathbb{R})$ with $n = 1 +k + \sum_{i=1}^k (|B_i|-1)$, and now we will form a $(u+1)\times (u+1)$ Jacobian matrix. Consider an arbitrary arm; some number of distinct upward eigenvalues were initially assigned to this arm. We associate at most that many non-upward eigenvalues to this arm. We associate $u$ non-upward eigenvalue constraints to arms while respecting this limitation. Considering an arbitrary arm, $B_i$, to which we have associated $\ell_i$ eigenvalue constraints. 

We then form the Jacobian with respect to $a_i$ and $\ell_i-1$ eigenvalues of $B_i(1)$, which without loss of generality we relabel as $\{\eta^{(i)}_1,\ldots,\eta^{(i)}_{\ell_i-1}\}$. We also always form the Jacobian with respect to $a_0$, which accounts for the last eigenvalue constraint. \textit{This gives us the $u+1$ variables with which we form the Jacobian}.

Our Jacobian is of the initial form

\[
J:= \kbordermatrix{
    & a_0 & & a_1 \thinspace\thinspace\thinspace B_1(1) &  & &   &a_k \thinspace\thinspace\thinspace B_k(1)\\
    \lambda_1 & -\prod_{j=1}^k p_{B_j}(\lambda_1)  &\vrule &  & \vrule &  & \vrule \\
    \lambda_3 & -\prod_{j=1}^k p_{B_j}(\lambda_3)&\vrule & & \vrule& &\vrule \\
    \vdots & \vdots  &\vrule & M_1 & \vrule & \cdots & \vrule &M_k \\
    \vdots & \vdots  &\vrule & & \vrule & &\vrule \\
    \lambda_{2u+1} & -\prod_{j=1}^k p_{B_j}(\lambda_{2u+1})  &\vrule  & &\vrule&&\vrule \\
  },
\]
in which each $M_i$ is the following $(u+1)\times \ell_i$ matrix

\[
M_i= \kbordermatrix{
     & a_1 &  &  B_i(1) & & \\
    \lambda_1   & -2a_1 p_{B_i(1)}(\lambda_1)\prod_{ \substack{\ell=1 \\ \ell \neq i }  }^k p_{B_\ell}(\lambda_1) & - \frac{a_1^2 p_{B_i(1)}(\lambda_1)}{\eta^{(i)}_1-\lambda_1}\prod_{ \substack{\ell=1 \\ \ell \neq i}  }^k p_{B_\ell}(\lambda_1)  & \cdots &- \frac{a_1^2 p_{B_i(\ell)}(\lambda_1)}{\eta^{(i)}_{\ell_i-1} - \lambda_1} (\lambda_1)\prod_{ \substack{\ell=1 \\ \ell \neq i}  }^k p_{B_\ell}(\lambda_1)   \\
    \lambda_3 & \vdots & \vdots & \vdots & \vdots \\
    \vdots   & \vdots & \vdots & \vdots & \vdots \\
    \lambda_{2u+1}   & \vdots & \vdots & \vdots & \vdots \\
  }.
\]

Each subsequent row will be a copy of the first, with $\lambda_1$ replaced by $\lambda_j, j=3,5,\ldots,2u+1$. Because of this, for the rest of this proof we will abuse notation and describe all the column operations by their action on the first row. We also let $\sim$ denote an operation which preserves nonsingularity of the Jacobian.

Note that \textit{a priori} each $M_i$ is a rectangular matrix without a well-defined determinant. However, for simplicity we will refer to its determinant, which refers instead to the determinant of the full Jacobian $J$. Since multiplying rows and columns by nonzero constants preserves singularity and nonsingularity, we divide each row in our Jacobian by $\prod_{ \substack{\ell=1 }  }^k p_{B_\ell}(\lambda_j)$. We then divide out by nonzero prefactors such as negative signs, constants, and $a_i$. This gives us the renormalized (square) Jacobian 
\[
 \kbordermatrix{
    & a_0 & & a_1 \thinspace\thinspace\thinspace B_1(1) &  & &   &a_k \thinspace\thinspace\thinspace B_k(1)\\
    \lambda_1 & 1 &\vrule &  & \vrule &  & \vrule \\
    \lambda_3 & 1 &\vrule & & \vrule& &\vrule \\
    \vdots & \vdots  &\vrule & \widetilde{M_1} & \vrule & \cdots & \vrule &\widetilde{M_k} \\
    \vdots & \vdots  &\vrule & & \vrule & &\vrule \\
    \lambda_{2u+1} & 1  &\vrule  & &\vrule&&\vrule \\
  },
\]
with
\[
\widetilde{M_i}= \kbordermatrix{
     & a_1 &\vrule &  &  B_i(1) & & \\
    \lambda_1   & \frac{ p_{B_i(1)}(\lambda_1)}{ p_{B_i}(\lambda_1) }  &\vrule &\frac{1}{\lambda_1-\eta^{(i)}_1} \frac{ p_{B_i(1)}(\lambda_1)}{ p_{B_i}(\lambda_1) }  & \cdots & \frac{ 1}{ \lambda_1- \eta^{(i)}_{l_i-1} }  \frac{ p_{B_i(1)}(\lambda_1)}{ p_{B_i}(\lambda_1) }   \\
    \lambda_3 & \vdots &\vrule& \vdots & \vdots & \vdots \\
    \vdots   & \vdots  &\vrule&\vdots & \vdots & \vdots \\
    \lambda_{2u+1}   & \frac{ p_{B_i(1)}(\lambda_{2u+1})}{ p_{B_i}(\lambda_{2u+1}) }  &\vrule &\frac{1}{\lambda_{2u+1}-\eta^{(i)}_1} \frac{ p_{B_i(1)}(\lambda_{2u+1})}{ p_{B_i}(\lambda_{2u+1}) }  & \cdots & \frac{ 1}{\lambda_{2u+1}-\eta^{(i)}_{\ell_i-1} }  \frac{ p_{B_i(1)}(\lambda_{2u+1})}{ p_{B_i}(\lambda_{2u+1}) }   \\
  }.
\]

The ratio of characteristic polynomials simplifies as
$$\frac{ p_{B_i(1)}(\lambda)}{ p_{B_i}(\lambda) }  = \frac{ \prod_{k=1}^{n_i-1} \left(\lambda- \eta^{(i)}_k\right)   }{\prod_{k=1}^{n_i} \left(\lambda- \hat{\mu}^{(i)}_k\right) }. $$
We now denote (note the different domain of the product) $$ \beta_i(\lambda) :=    \frac{ \prod_{k=\ell_i}^{n_i-1} \left(\lambda- \eta^{(i)}_k\right)   }{\prod_{k=\ell_1+1}^{n_i} \left(\lambda- \hat{\mu}^{(i)}_k\right) }.$$ 

\textbf{Example}. A matrix with the shape of $\mathcal{T}_7$ has the form
$$
\begin{bmatrix}
   {\color{blue}a_{0} }& {\color{blue}a_{1} } & {\color{blue}a_{2} } & 0 & {\color{blue}a_{3} } & 0 &0 \\
   {\color{blue}a_{1} } & \hat{\mu}_2 & 0 & 0 & 0 & 0 &0 \\
   {\color{blue}a_{2} } & 0 & \hat{\mu}_2 & {\alpha } & 0 & 0 &0 \\
   {0} & 0 & {\alpha } & \hat{\mu}_6 & 0 & 0 &0 \\
   {\color{blue}a_{3} } & 0 & 0 & 0 & \hat{\mu}_2 & {\color{blue}\beta } &0 \\
   {0} & 0 & 0 & 0 & {\color{blue}\beta } & \hat{\mu}_4 &{\gamma } \\
   {0} & 0 & 0 & 0 & 0 & {\gamma } &\hat{\mu}_8 \\
\end{bmatrix},
$$
with blue text corresponding to variables we construct a Jacobian with respect to, in terms of matrix entries. Note that each arm cannot be assigned an upwards eigenvalue with multiplicity greater than one, by Theorem \ref{gstarmult}. In terms of the underlying assignment of Figure \ref{fig2}, each $\hat{\mu}_{2i}$ can be associated with a unique arm.

We now compose this map with the smooth change of variables $(\alpha, {\color{blue}\beta}, \gamma) \to (\eta^{(2)}_1,{\color{blue}\eta^{(3)}_1},\eta^{(3)}_2)$, mapping subdiagonal entries to eigenvalues of the principal submatrix, and will construct a Jacobian with respect to $\{{\color{blue}a_0,a_1,a_2,a_3,\eta^{(3)}_1 }\}$. In terms of eigenvalues, the characteristic polynomials expands as 
\begin{align*}
    p(t) &= (t-{\color{blue}a_{0} }) (t-\hat{\mu}_2)^3(t-\hat{\mu}_4)(t-\hat{\mu}_6)(t-\hat{\mu}_8) \\
    &\quad -{\color{blue}a_{1}}^2 (t-\hat{\mu}_2)^2(t-\hat{\mu}_4)(t-\hat{\mu}_6)(t-\hat{\mu}_8)  \\
    &\quad -{\color{blue}a_{2} }^2 (t-\eta_1^{(2)})(t-\hat{\mu}_2)(t-\hat{\mu}_6) \\
    &\quad -{\color{blue}a_{3} }^2 (t-{\color{blue}\eta_1^{(3)}})(t-\eta_2^{(3)})(t-\hat{\mu}_2)^2(t-\hat{\mu}_4)(t-\hat{\mu}_8).
    \end{align*}
We have fixed $\{\hat{\mu}_2,\hat{\mu}_4,\hat{\mu}_6,\hat{\mu}_8 \}$, as this is part of our desired spectrum, so we cannot perturb these eigenvalues.

Label the shortest arm as arm $1$, the length $2$ arm as $2$, and the length $3$ arm as $3$. We then have several parameters attached to each arm:
\begin{itemize}
\item Arm 1: $n_1=\ell_1=1, \beta_1(\lambda) = 1$ (empty product);
\item Arm 2: $n_2=2, \ell_2=1, \beta_2(\lambda) = \frac{\lambda-\eta^{(2)}_1}{\lambda-\hat{\mu}_2}$;
\item Arm 3: $n_3=3, \ell_3=2, \beta_3(\lambda) = \frac{\lambda-\eta^{(3)}_2}{\lambda-\hat{\mu}_2}$.
\end{itemize}
We can then express the Jacobian of $\mathcal{T}_7$ in terms of eigenvalues as

\begin{align*}
& \kbordermatrix{
     &  &  \text{Arm 1}& \text{Arm 2} & \hspace{3cm} \text{Arm 3}  & \\
    \lambda_1   & -1  & \frac{-2a_1}{\lambda_1 - \hat{\mu}_2} & \frac{-2a_2(\lambda_1-\eta^{(2)}_1)}{(\lambda_1 - \hat{\mu}_2)(\lambda_1 - \hat{\mu}_6)} & \frac{-2a_3(\lambda_1-\eta^{(3)}_1)(\lambda_1-\eta^{(3)}_2)}{(\lambda_1 - \hat{\mu}_2)(\lambda_1 - \hat{\mu}_4)(\lambda_1 - \hat{\mu}_8)} & \frac{(\lambda_1-\eta^{(3)}_2)}{(\lambda_1 - \hat{\mu}_2)(\lambda_1 - \hat{\mu}_4)(\lambda_1 - \hat{\mu}_8)}\\
    \lambda_3 & \vdots & \vdots & \vdots & \vdots & \vdots \\
    \lambda_5 & \vdots & \vdots & \vdots & \vdots & \vdots \\
    \lambda_7 & \vdots & \vdots & \vdots & \vdots & \vdots \\
    \lambda_{9} & \vdots  &\cdots & \cdots & \cdots & \vdots \\
  }  \\
  &\sim_{\text{divide out constants}} \kbordermatrix{
     &  & &  &   & \\
    & 1  & \frac{1}{\lambda_1 - \hat{\mu}_2} & \frac{(\lambda_1-\eta^{(2)}_1)}{(\lambda_1 - \hat{\mu}_2)(\lambda_1 - \hat{\mu}_6)} & \frac{(\lambda_1-\eta^{(3)}_1)(\lambda_1-\eta^{(3)}_2)}{(\lambda_1 - \hat{\mu}_2)(\lambda_1 - \hat{\mu}_4)(\lambda_1 - \hat{\mu}_8)} & \frac{(\lambda_1-\eta^{(3)}_2)}{(\lambda_1 - \hat{\mu}_2)(\lambda_1 - \hat{\mu}_4)(\lambda_1 - \hat{\mu}_8)}\\
  }  \\
&=_{\text{definition of $\beta$}} \kbordermatrix{
         &  & &  &   & \\
       & 1  & \frac{1}{\lambda_1 - \hat{\mu}_2} & \beta_2(\lambda_1)\frac{1}{\lambda_1 - \hat{\mu}_6} & \beta_3(\lambda_1)\frac{(\lambda_1-\eta^{(3)}_1)}{(\lambda_1 - \hat{\mu}_4)(\lambda_1 - \hat{\mu}_8)}& \beta_3(\lambda_1)\frac{1}{(\lambda_1 - \hat{\mu}_4)(\lambda_1 - \hat{\mu}_8)}\\
  } . \\
\end{align*}

Each other row of our Jacobian contains the same expression with $\lambda_1$ replaced by $\lambda_3, \lambda_5, \lambda_7, \lambda_9. $

\textbf{Claim 1.}
Rectangular matrix $\widetilde{M_i}$ has $\ell_i$ columns. We can reduce the numerator of the non-$\beta$ factor in column $j$ to $\prod_{k=j}^{\ell_i-1} \left(\lambda_i - \eta_k^{(i)}\right)$ by sweeping right.

We take the second column in $\widetilde{M_i}$ and subtract it from each subsequent column (but not the first). For column $j$, we obtain 
$$\beta_i(\lambda) \prod_{\substack{k=1 \\ k \neq j}}^{\ell_i-1 } \left(\lambda- \eta^{(i)}_k\right) - 
\beta_i(\lambda) \prod_{\substack{k=1 \\ k \neq 1}}^{\ell_i-1 } \left(\lambda- \eta^{(i)}_k\right)  =
  \left( \eta_1^{(i)}-\eta_j^{(i)}\right) \beta_i(\lambda) \prod_{\substack{k=1 \\ k \neq 1, j}}^{\ell_i-1 } \left(\lambda- \eta^{(i)}_k\right). $$
Since $ \left( \eta_1^{(i)}-\eta_j^{(i)}\right)$ is a constant down the column, and is nonzero since the $\eta$'s \textit{strictly} interlace the $\hat{\mu}$'s, we can factor $\prod_{k=1}^{\ell_i-1}\left( \eta_1^{(i)}-\eta_k^{(i)}\right)$ from the entire determinant while preserving singularity or nonsingularity. Note that we cannot factor out $\beta_i$ since it explicitly depends on $\lambda$. Iterating this process by subtracting the new third column from each subsequent column, and so on, yields the following simplified submatrix:

\[
\widetilde{M_i}\sim \kbordermatrix{
     & a_1 &  &  B_i(1) & & \\
    \lambda_1   & \beta_i(\lambda_1)  \frac{ \prod_{k=1}^{\ell_i-1} \left(\lambda_1- \eta^{(i)}_k\right)   }{\prod_{k=1}^{\ell_i} \left(\lambda_1- \hat{\mu}^{(i)}_k\right) }        & \beta_i(\lambda_1)  \frac{ \prod_{k=2}^{\ell_i-1} \left(\lambda_1- \eta^{(i)}_k\right)   }{\prod_{k=1}^{\ell_i} \left(\lambda_1- \hat{\mu}^{(i)}_k\right) }  & \cdots & \beta_i(\lambda_1)  \frac{1   }{\prod_{k=1}^{\ell_i} \left(\lambda_1- \hat{\mu}^{(i)}_k\right) }\\
    \lambda_3 & \vdots & \vdots & \vdots & \vdots \\
    \vdots   & \vdots  &\vdots & \vdots & \vdots \\
    \lambda_{2u+1} & \vdots  &\vdots & \cdots & \vdots \\
  }.
\]

\textbf{Example}. 
Our example $\mathcal{T}_7$ is simple enough that we do not have to perform any sweeps right. However, the operations performed on $\mathcal{T}_7$ in Claim $2$ highlight the same procedure.

\textbf{Claim 2}.
Rectangular matrix $\widetilde{M_i}$ has $\ell_i$ columns. We can reduce the numerator of the non-$\beta$ factor in column $j$ in $\widetilde{M_i}$ to $\lambda^{\ell_i-j}$ by repeated sweeps left.

Regarding the numerators as polynomials in $\lambda$, the right-most column has numerator degree $0$, the next has degree $1$, and so on. Hence we can sweep left, subtracting appropriate multiples of the right-most column from every other column so that the constant terms in the numerator are zero. We can iterate this process, subtracting appropriate multiples of the second right-most column from column to its left. In general, everything except the leading term can be cancelled, leaving

\[
\widetilde{M_i}\sim \kbordermatrix{
   & \\
    \lambda_1   & \beta_i(\lambda_1)  \frac{ \lambda_1^{\ell_1-1} }{\prod_{k=1}^{\ell_i} \left(\lambda_1- \hat{\mu}^{(i)}_k\right) }        & \beta_i(\lambda_1)  \frac{\lambda_1^{\ell_1-2} }{\prod_{k=1}^{\ell_i} \left(\lambda_1- \hat{\mu}^{(i)}_k\right) }  & \cdots & \beta_i(\lambda_1)  \frac{ 1   }{\prod_{k=1}^{\ell_i} \left(\lambda_1- \hat{\mu}^{(i)}_k\right) }\\
    \lambda_3 & \vdots & \vdots & \vdots & \vdots \\
    \vdots   & \vdots  &\vdots & \vdots & \vdots \\
    \lambda_{2u+1} & \vdots  &\vdots & \cdots & \vdots \\
  }.
\]
We can now sweep left again, creating an arbitrary $\lambda$ polynomial of degree $j$ in column $j$ by adding appropriate multiples of columns to the right. We will in fact put $  \prod_{k=1}^{\ell_i-j} \left(\lambda_1- \hat{\mu}^{(i)}_k\right) $ in the numerator, which will cancel with partial denominator products. This finally gives the much simpler matrix
\[
\widetilde{M_i}\sim \kbordermatrix{
   & \\
    \lambda_1   &   \frac{ \beta_i(\lambda_1) }{\prod_{k=\ell_i}^{\ell_i} \left(\lambda_1- \hat{\mu}^{(i)}_k\right) }        &   \frac{\beta_i(\lambda_1) }{\prod_{k=\ell_i-1}^{\ell_i} \left(\lambda_1- \hat{\mu}^{(i)}_k\right) }  & \cdots &  \frac{ \beta_i(\lambda_1)   }{\prod_{k=1}^{\ell_i} \left(\lambda_1- \hat{\mu}^{(i)}_k\right) }\\
    \lambda_3 & \vdots & \vdots & \vdots & \vdots \\
    \vdots   & \vdots  &\vdots & \vdots & \vdots \\
    \lambda_{2u+1} & \vdots  &\vdots & \cdots & \vdots \\
  }.
\]

\textbf{Example}. By repeated sweeps left, we reduce the Jacobian for $\mathcal{T}_7$ to 

\begin{align*}
&\kbordermatrix{
     &  &  \text{Arm 1}& \text{Arm 2} & \hspace{3cm} \text{Arm 3}  & \\
       & 1  & \frac{1}{\lambda_1 - \hat{\mu}_2} & \beta_2(\lambda_1)\frac{1}{\lambda_1 - \hat{\mu}_6} & \beta_3(\lambda_1)\frac{(\lambda_1-\eta^{(3)}_1)}{(\lambda_1 - \hat{\mu}_4)(\lambda_1 - \hat{\mu}_8)}& \beta_3(\lambda_1)\frac{1}{(\lambda_1 - \hat{\mu}_4)(\lambda_1 - \hat{\mu}_8)}\\
  } . \\
 &\sim_{\text{sweep left}} \kbordermatrix{
     &  & & &  & \\
       & 1  & \frac{1}{\lambda_1 - \hat{\mu}_2} & \beta_2(\lambda_1)\frac{1}{\lambda_1 - \hat{\mu}_6} & \beta_3(\lambda_1)\frac{(\lambda_1-\eta^{(3)}_1)}{(\lambda_1 - \hat{\mu}_4)(\lambda_1 - \hat{\mu}_8)}+\eta_1^{(3)}\beta_3(\lambda_1)\frac{1}{(\lambda_1 - \hat{\mu}_4)(\lambda_1 - \hat{\mu}_8)}& \beta_3(\lambda_1)\frac{1}{(\lambda_1 - \hat{\mu}_4)(\lambda_1 - \hat{\mu}_8)}\\
}.\\
&= \kbordermatrix{
     &  & & &  & \\
       & 1  & \frac{1}{\lambda_1 - \hat{\mu}_2} & \beta_2(\lambda_1)\frac{1}{\lambda_1 - \hat{\mu}_6} & \beta_3(\lambda_1)\frac{\lambda_1}{(\lambda_1 - \hat{\mu}_4)(\lambda_1 - \hat{\mu}_8)}& \beta_3(\lambda_1)\frac{1}{(\lambda_1 - \hat{\mu}_4)(\lambda_1 - \hat{\mu}_8)}\\
}.\\
&\sim_{\text{sweep left}} \kbordermatrix{
     &  & & &  & \\
       & 1  & \frac{1}{\lambda_1 - \hat{\mu}_2} & \beta_2(\lambda_1)\frac{1}{\lambda_1 - \hat{\mu}_6} & \beta_3(\lambda_1)\frac{\lambda_1}{(\lambda_1 - \hat{\mu}_4)(\lambda_1 - \hat{\mu}_8)} -\hat{\mu}_4\beta_3(\lambda_1)\frac{1}{(\lambda_1 - \hat{\mu}_4)(\lambda_1 - \hat{\mu}_8)}& \beta_3(\lambda_1)\frac{1}{(\lambda_1 - \hat{\mu}_4)(\lambda_1 - \hat{\mu}_8)}\\
}.\\
&= \kbordermatrix{
     &  & & &  & \\
       & 1  & \frac{1}{\lambda_1 - \hat{\mu}_2} & \beta_2(\lambda_1)\frac{1}{\lambda_1 - \hat{\mu}_6} & \beta_3(\lambda_1)\frac{1}{(\lambda_1 - \hat{\mu}_8)}& \beta_3(\lambda_1)\frac{1}{(\lambda_1 - \hat{\mu}_4)(\lambda_1 - \hat{\mu}_8)}\\
}.\\
\end{align*}

\textbf{Claim 3}.
We can reduce the degree of the denominator of the non-$\beta_i$ factors to $1$ by repeated sweeps right.

In column $j$, we now employ the partial fraction decomposition
$$ \frac{\beta_i(\lambda_1) }{\prod_{k=\ell_i-j+1}^{\ell_i} \left(\lambda_1- \hat{\mu}^{(i)}_k\right) }=  \beta_i(\lambda_1) \sum_{k=\ell_1-j+1}^{\ell_i}\frac{ 1}{\prod_{\substack{m=\ell_i-j+1 \\ m\neq k}}^{\ell_i} \left(\hat{\mu}^{(i)}_k- \hat{\mu}^{(i)}_m\right) } \frac{1}{ \lambda_1- \hat{\mu}^{(i)}_k},$$
where each constant in front of $ \frac{1}{ \lambda_1- \hat{\mu}^{(i)}_k}$ is nonzero since an upward eigenvalue cannot be assigned to an arm twice. Since the left-most column has a single term in its partial fraction expansion, we can use it to eliminate any subsequent occurrence of $ \frac{1}{ \lambda_1- \hat{\mu}^{(i)}_1}$ by subtracting appropriate multiples of the first column from subsequent columns. Iterating this process from left to right and dividing out by the nonzero constants in the partial fraction expansion gives the further simplified matrix 
\[
\widetilde{M_i}\sim \kbordermatrix{
   & \\
    \lambda_1   &   \frac{ \beta_i(\lambda_1) }{\lambda_1- \hat{\mu}^{(i)}_{\ell_i} }     &   \frac{ \beta_i(\lambda_1) }{\lambda_1- \hat{\mu}^{(i)}_{\ell_i-1} }    & \cdots &   \frac{ \beta_i(\lambda_1) }{\lambda_1- \hat{\mu}^{(i)}_{1} }   \\
    \lambda_3 & \vdots & \vdots & \vdots & \vdots \\
    \vdots   & \vdots  &\vdots & \vdots & \vdots \\
    \lambda_{2u+1} & \vdots  &\vdots & \cdots & \vdots \\
  }.
\]

\textbf{Example}. For $\mathcal{T}_7$, we obtain
\begin{align*}
    & \kbordermatrix{
     &  &  \text{Arm 1}& \text{Arm 2} & \hspace{3cm} \text{Arm 3}  & \\
       & 1  & \frac{1}{\lambda_1 - \hat{\mu}_2} & \beta_2(\lambda_1)\frac{1}{\lambda_1 - \hat{\mu}_6} & \beta_3(\lambda_1)\frac{1}{(\lambda_1 - \hat{\mu}_8)}& \beta_3(\lambda_1)\frac{1}{(\lambda_1 - \hat{\mu}_4)(\lambda_1 - \hat{\mu}_8)}\\
}.\\
&=_{\text{partial fraction}} \kbordermatrix{
     &  &  &  &  & \\
       & 1  & \frac{1}{\lambda_1 - \hat{\mu}_2} & \beta_2(\lambda_1)\frac{1}{\lambda_1 - \hat{\mu}_6} & \beta_3(\lambda_1)\frac{1}{(\lambda_1 - \hat{\mu}_8)}& \frac{\beta_3(\lambda_1)}{(\hat{\mu_4}-\hat{\mu_8})}\left[\frac{1}{(\lambda_1 - \hat{\mu_4})}-\frac{1}{(\lambda_1 - \hat{\mu_8})} \right]\\
}.\\
&\sim_{\text{divide out}} \kbordermatrix{
     &  &  &  &  & \\
       & 1  & \frac{1}{\lambda_1 - \hat{\mu}_2} & \beta_2(\lambda_1)\frac{1}{\lambda_1 - \hat{\mu}_6} & \beta_3(\lambda_1)\frac{1}{(\lambda_1 - \hat{\mu}_8)}& \beta_3(\lambda_1)\left[\frac{1}{(\lambda_1 - \hat{\mu_4})}-\frac{1}{(\lambda_1 - \hat{\mu_8})} \right]\\
}.\\
&\sim_{\text{sweep right}} \kbordermatrix{
     &  &  &  &  & \\
       & 1  & \frac{1}{\lambda_1 - \hat{\mu}_2} & \beta_2(\lambda_1)\frac{1}{\lambda_1 - \hat{\mu}_6} & \beta_3(\lambda_1)\frac{1}{(\lambda_1 - \hat{\mu}_8)}& \beta_3(\lambda_1)\frac{1}{(\lambda_1 - \hat{\mu_4})}\\
}.\\
\end{align*}

\textbf{Claim 4}. We can eliminate the $\beta_i$ terms while preserving nonsingularity.

We finally stop regarding $\beta_i(\lambda_1)$ as a constant and partial fraction each entry again, giving 
$$ \frac{\beta_i(\lambda_1) }{\lambda_1- \hat{\mu}^{(i)}_j }=   \sum_{k=\ell_i+1}^{n_i}\frac{ \prod_{\substack{m=\ell_i }}^{n_i-1} \left(\hat{\mu}^{(i)}_k- \eta^{(i)}_m\right) }{\prod_{\substack{m=\ell_i+1 \\ m\neq k}}^{n_i} \left(\hat{\mu}^{(i)}_k- \hat{\mu}^{(i)}_m \right) } \frac{1}{\hat{\mu}^{(i)}_k - \hat{\mu}^{(i)}_j} \frac{1}{ \lambda_1- \hat{\mu}^{(i)}_k} +  
\frac{ \prod_{\substack{m=\ell_i }}^{n_i-1} \left(\hat{\mu}^{(i)}_j- \eta^{(i)}_m\right) }{\prod_{\substack{m=\ell_i+1 \\ m\neq k}}^{n_i} \left(\hat{\mu}^{(i)}_j- \hat{\mu}^{(i)}_m \right) } \frac{1}{ \lambda_1- \hat{\mu}^{(i)}_j}  ,$$
where each constant in the expansion is nonzero by strict interlacing and uniqueness of the eigenvalues under consideration. We now consider the entire Jacobian
\[
 \kbordermatrix{
    & a_0 & & a_1 \thinspace\thinspace\thinspace B_1(1) &  & &   &a_k \thinspace\thinspace\thinspace B_k(1)\\
    \lambda_1 & 1 &\vrule &  & \vrule &  & \vrule \\
    \lambda_3 & 1 &\vrule & & \vrule& &\vrule \\
    \vdots & \vdots  &\vrule & \widetilde{M_1} & \vrule & \cdots & \vrule &\widetilde{M_k} \\
    \vdots & \vdots  &\vrule & & \vrule & &\vrule \\
    \lambda_{2u+1} & 1  &\vrule  & &\vrule&&\vrule \\
  }.
\]
 As previously noted, in terms of the underlying assignment of Figure \ref{fig2}, each $\hat{\mu}_{2i}$ can be associated with a unique $\lambda_{2i+1}$. This ultimately follows from Theorem \ref{gstarmult}, which gives the allowable spectra for generalized stars. There are $u$ of these associations, along with another column of all $1$s, accounting for all $u+1$ columns. Then, $\{1, \frac{1}{\lambda - \hat{\mu}_{2}}, \frac{1}{\lambda - \hat{\mu}_{4}}, \ldots, \frac{1}{\lambda - \hat{\mu}_{2u}}\}$ forms a basis for the column space, and the nonsingularity of our Jacobian is equivalent to the nonsingularity of the following $(u+1)\times (u+1)$ matrix:

\[
 \kbordermatrix{
    & a_0 &  \hat{\mu}_2& \hat{\mu}_4 & \ldots & \hat{\mu}_{2u}\\
    \lambda_1 & 1 & \frac{1}{\lambda_1 -\hat{\mu}_2} &  \frac{1}{\lambda_1 -\hat{\mu}_4} &\cdots & \frac{1}{\lambda_1 -\hat{\mu}_{2u}}\\
    \lambda_3 & 1 & \vdots & & & \vdots \\
    \vdots & 1 & \vdots & & & \vdots \\
    \lambda_{2u+1} & 1 & \frac{1}{\lambda_{2u+1} -\hat{\mu}_2} &  \frac{1}{\lambda_{2u+1} -\hat{\mu}_4} &\cdots & \frac{1}{\lambda_{2u+1} -\hat{\mu}_{2u}}\\
  }.
\]
\textbf{Example}. As previously noted, in terms of the underlying assignment of Figure \ref{fig2}, each $\hat{\mu}_{2i}$ can be associated with a unique column. This ultimately follows from Theorem \ref{gstarmult}, the allowable spectra for generalized stars. This accounts for the line where we associate a unique $\hat{\mu}_{2i}$. For $\mathcal{T}_7$, this means that we can create a correspondence between (Arm 1, Arm 2, Arm 3) $ \leftrightarrow (\hat{\mu}_{2}, \hat{\mu}_{6}, \{\hat{\mu}_{4},\hat{\mu}_{8}\} )$.  This accounts for the line where we associate a unique $\hat{\mu}_{2i}$ to each column. All other terms will be eliminated by taking the correct linear combinations of columns.

\begin{align*}
&\kbordermatrix{
     &  &  \text{Arm 1}& \text{Arm 2} & \hspace{3cm} \text{Arm 3}  & \\
       & 1  & \frac{1}{\lambda_1 - \hat{\mu}_2} & \beta_2(\lambda_1)\frac{1}{\lambda_1 - \hat{\mu}_6} & \beta_3(\lambda_1)\frac{1}{(\lambda_1 - \hat{\mu}_8)}& \beta_3(\lambda_1)\frac{1}{(\lambda_1 - \hat{\mu_4})}\\
}.\\
&=_{\text{definition of $\beta_i$}}\kbordermatrix{
     &  & & &  & \\
       & 1  & \frac{1}{\lambda_1 - \hat{\mu}_2} & \frac{\lambda_1-\eta^{(2)}_1}{\lambda_1-\hat{\mu}_2}\frac{1}{\lambda_1 - \hat{\mu}_6} & \frac{\lambda_1-\eta^{(3)}_2}{\lambda_1-\hat{\mu}_2}\frac{1}{\lambda_1 - \hat{\mu}_8}& \frac{\lambda_1-\eta^{(3)}_2}{\lambda_1-\hat{\mu}_2}\frac{1}{\lambda_1 - \hat{\mu}_4}\\
}.\\
&=_{\text{partial fraction}}\kbordermatrix{
     &  & & &  & \\
       & 1  & \frac{1}{\lambda_1 - \hat{\mu}_2} & \frac{1}{\hat{\mu}_2-\hat{\mu}_6}\left(\frac{\hat{\mu}_2-\eta^{(2)}_1}{\lambda_1 - \hat{\mu}_2}-\frac{ \hat{\mu}_6-\eta^{(2)}_1 }{\lambda_1 - \hat{\mu}_6}\right) & \frac{1}{\hat{\mu}_2-\hat{\mu}_8}\left(\frac{\hat{\mu}_2-\eta^{(3)}_2}{\lambda_1 - \hat{\mu}_2}-\frac{ \hat{\mu}_8-\eta^{(3)}_2 }{\lambda_1 - \hat{\mu}_6}\right)& \frac{1}{\hat{\mu}_2-\hat{\mu}_4}\left(\frac{\hat{\mu}_2-\eta^{(3)}_2}{\lambda_1 - \hat{\mu}_2}-\frac{ \hat{\mu}_8-\eta^{(3)}_2 }{\lambda_1 - \hat{\mu}_4}\right)\\
}.\\
&\sim_{\text{unique $\hat{\mu}_{2i}$ per column}}\kbordermatrix{
     &  & & &  & \\
       & 1  & \frac{1}{\lambda_1 - \hat{\mu}_2} & 
       -\frac{ \hat{\mu}_6-\eta^{(2)}_1 }{\hat{\mu}_2 - \hat{\mu}_6} \frac{1}{ \lambda_1 -\hat{\mu}_6 }
        & -\frac{ \hat{\mu}_8-\eta^{(3)}_2 }{\hat{\mu}_2 - \hat{\mu}_8} \frac{1}{ \lambda_1 -\hat{\mu}_8 }& -\frac{ \hat{\mu}_4-\eta^{(3)}_2 }{\hat{\mu}_2 - \hat{\mu}_4} \frac{1}{ \lambda_1 -\hat{\mu}_4 }\\
}.\\
&\sim_{\text{divide out constants}}\kbordermatrix{
     &  & & &  & \\
       & 1  & \frac{1}{\lambda_1 - \hat{\mu}_2} & \frac{1}{\lambda_1 - \hat{\mu}_6} & \frac{1}{\lambda_1 - \hat{\mu}_8}& \frac{1}{\lambda_1 - \hat{\mu}_4}\\
}.\\
&\sim_{\text{permute columns}}\kbordermatrix{
     &  & & &  & \\
       & 1  & \frac{1}{\lambda_1 - \hat{\mu}_2} & \frac{1}{\lambda_1 - \hat{\mu}_4} & \frac{1}{\lambda_1 - \hat{\mu}_6}& \frac{1}{\lambda_1 - \hat{\mu}_8}\\
}.\\
\end{align*}


\textbf{Claim 5}.
The Jacobian is nonsingular.

This determinant is a bordered variant of Cauchy's \textit{double alternant} \cite{Krattenthaler}. The following rescaled determinant has already been explicitly computed \cite{TwoAndrews}: 
$$
\det \begin{bmatrix}
    \prod_{   i=1}^{n-1}(a_i-x_1) & \prod_{\substack{i=1\\i\neq 1}}^{n-1}(a_i-x_1)  & \cdots  & \prod_{\substack{i=1\\i\neq n-1}}^{n-1}(a_i-x_1)  \\
    \vdots & \vdots  &  & \vdots \\
    \vdots & \vdots  &  & \vdots \\
    \prod_{   i=1}^{n-1}(a_i-x_n) & \prod_{\substack{i=1\\i\neq 1}}^{n-1}(a_i-x_n)  & \cdots  & \prod_{\substack{i=1\\i\neq n-1}}^{n-1}(a_i-x_n)  \\
\end{bmatrix} = \prod_{1\leq i < j \leq n-1}(a_i-a_j)  \prod_{1\leq i < j \leq n}(x_j-x_i).
$$
Since we have $\lambda_1< \hat{\mu}_2 < \lambda_2 < \cdots < \hat{\mu}_{2u} < \lambda_{2u+1}$ with \textit{strict} interlacing, we finally have that 
\begin{align*}
\det &\begin{bmatrix}
   1 & \frac{1}{\lambda_1 -\hat{\mu}_2} &  \frac{1}{\lambda_1 -\hat{\mu}_4} &\cdots & \frac{1}{\lambda_1 -\hat{\mu}_{2u}}\\
    1 & \vdots & & & \vdots \\
     1 & \vdots & & & \vdots \\
     1 & \frac{1}{\lambda_{2u+1} -\hat{\mu}_2} &  \frac{1}{\lambda_{2u+1} -\hat{\mu}_4} &\cdots & \frac{1}{\lambda_{2u+1} -\hat{\mu}_{2u}}\\
\end{bmatrix}   \\
&= \frac{(-1)^u}{\prod_{j=1}^{u+1}\prod_{i=1}^u (\lambda_{2j-1}-\hat{\mu}_{2i})}
\det \begin{bmatrix}
    \prod_{   i=1}^{u}(\hat{\mu}_{2i}-\lambda_1) & \prod_{\substack{i=1\\i\neq 1}}^{u}(\hat{\mu}_{2i}-\lambda_1)  & \cdots  & \prod_{\substack{i=1\\i\neq u}}^{u}(\hat{\mu}_{2i}-\lambda_1)  \\
    \vdots & \vdots  &  & \vdots \\
    \vdots & \vdots  &  & \vdots \\
    \prod_{   i=1}^{u}(\hat{\mu}_{2i}-\lambda_{2u+1}) & \prod_{\substack{i=1\\i\neq 1}}^{u}(\hat{\mu}_{2i}-\lambda_{2u+1})  & \cdots  & \prod_{\substack{i=1\\i\neq u}}^{u}(\hat{\mu}_{2i}-\lambda_{2u+1}) 
\end{bmatrix}   \\
&= \frac{(-1)^u \prod_{1\leq i < j \leq u}(\hat{\mu}_{2i}-\hat{\mu}_{2j})  \prod_{1\leq i < j \leq u+1}(\lambda_{2j-1}-\lambda_{2i-1})}{\prod_{j=1}^{u+1}\prod_{i=1}^u (\lambda_{2j-1}-\hat{\mu}_{2i})}\\
&\neq 0,
\end{align*}
so that our Jacobian is nonsingular.
\end{proof}

Using the implicit function theorem in this way, we may now give our primary result by verifying the sufficiency of the LSP for all linear trees. The necessity is from \cite{TwoAndrews}, and the sufficiency uses the implicit function theorem (in the same fashion as the partial result \cite[Theorem 19]{TwoAndrews}), but in a somewhat different way by viewing paths in the new way (spectrally) that we have described.
\begin{thm}\cite[Conjecture 10.2.3]{Johnson3}\label{lspthm}
An ordered multiplicity list occurs in $\mathcal{L}_0(T)$ for a linear tree $T$ if and only if that list can be obtained from the lists for the constituent paths and generalized stars of $T$ via the LSP.
\end{thm}
\begin{proof}
Necessity of the LSP was proven in \cite[Theorem 14]{TwoAndrews}. We have only provided a sketch of the proof of sufficiency, as it is exactly the proof of \cite[Theorem 19]{TwoAndrews}, with their use of \cite[Lemma 18]{TwoAndrews} replaced by our use of Lemma \ref{lem10}. The result in \cite[Theorem 19]{TwoAndrews} proves sufficiency for the depth $2$ case, where each arm is restricted to have length $1$.

For sufficiency, we use the implicit function theorem. Suppose the tabular form of the LSP has been completed for $T$, with valid multiplicity lists for the constituent paths and generalized stars of $T$. The initial matrix $A^{(0)}$ for our application of the implicit function theorem is the direct sum of several simpler matrices, those associated with the constituent paths and generalized stars of $T$ (whose multiplicity lists are those mentioned above). For each constituent generalized star, we know how to assign all relevant data (including any upward eigenvalues, of which there are none for paths). The edges connecting the constituent components correspond to the implicit function theorem variables. According to Lemma \ref{lem10}, the Jacobian at $A^{(0)}$ may be taken to be nonsingular. The conditions of Lemma \ref{lem8} have now been met. Hence, values of the connecting edge variables exist, according to the implicit function theorem, so that we may construct a matrix $A$ with graph $T$, whose spectrum is that given by the LSP.
\end{proof}
Since the eigenvalues used to verify the multiplicity lists in Theorem \ref{lspthm} were arbitrary, subject to the order of the eigenvalues, the corresponding inverse eigenvalue problem (IEP) is also solved. Note that this is not generally true for nonlinear trees \cite{Barioli}.
\begin{cor}[IEP for Linear Trees]
Given a linear tree $T$, an ordered list of multiplicities that occurs for $T$, and for any list of real eigenvalues respecting this ordered list, there is a real symmetric matrix in $\mathcal{S}(T)$ whose spectrum is the given list of eigenvalues.
\end{cor}

\section{Consequences of the LSP}\label{sec4}
The sufficiency of the LSP has a number of very interesting consequences, namely the Degree Conjecture, a result on subdivision of multiplicity lists, a result on augmentation of the underlying tree, a bound on the minimum number of ones among the multiplicity lists of $T$, and a formula for the maximum multiplicity of an eigenvalue in a multiplicity list. 


Let $\mathcal{L}(T)$ denote the set of unordered multiplicity lists for matrices whose graph is $T$. The Degree Conjecture asserts that a tree with $k$ high degree vertices with degrees $d_1,\ldots, d_k$ has $(d_1-1, d_2-1, \ldots, d_k-1, 1, 1, \ldots, 1) \in \mathcal{L}(T)$. 

\begin{cor}
The Degree Conjecture holds for linear trees.
\end{cor}
\begin{proof}
The work \cite{TwoAndrews} proves that any family of linear trees satisfying the sufficiency of the LSP conditions satisfies the Degree Conjecture. 
\end{proof}
We can also study the subdivision of multiplicity lists.
\begin{cor}
Let $T$ be a linear tree with $(m_1,m_2,\ldots,m_k)\in \mathcal{L}(T)$. Then, for any $j$ such that $m_j \geq 2, 1\leq j \leq k$, $$(m_1,\ldots,m_j-1,\ldots,m_k,1)\in \mathcal{L}(T).$$ 
\end{cor}
\begin{proof}
The paper \cite[Theorem 9]{Buckley} shows that this result follows for any linear tree satisfying the sufficiency of the LSP.
\end{proof}

We can prove a similar conjectured result concerning augmented multiplicity lists. Fix a tree $T$. Then we can augment $T$ to construct $T'$ in two ways: adding a pendant vertex or subdividing an edge by placing a vertex between its two endpoints. While constructing the set of ordered multiplicity lists $\mathcal{L}_0(T')$ is equivalent to the entire multiplicity list problem and an extremely subtle question, we can easily limit the set $\mathcal{L}_0(T')$. Let $\mathcal{L}^1_0(T)$ be the set of lists obtained by appending a $1$ to each list in $\mathcal{L}_0(T)$, and $\mathcal{L}^+_0(T)$ the set of lists obtained by adding a $1$ to each multiplicity list in $\mathcal{L}_0(T)$ in any possible way (including appending $1$'s). 
\begin{cor}\cite[Conjecture 11]{Buckley}
For any linear tree $T$, we have $$\mathcal{L}^1_0(T) \subseteq \mathcal{L}_0(T') \subseteq \mathcal{L}^+_0(T)$$ and
$$\mathcal{L}^1(T) \subseteq \mathcal{L}(T') \subseteq \mathcal{L}^+(T).$$
\end{cor}
\begin{proof}
The second claim follows from the first, since it simply ignores list orderings. Since $T$ is linear, $T'$ will also be linear, and will consist of $T$ with an extra vertex in some intermediate path or with a generalized star that has an arm elongated by $1$. 

Consider $L \in \mathcal{L}_0(T)$ and the underlying superposed lists that generate it. In the first instance, we construct an extra column in the tabular form of the LSP and insert a single $1$ in the row corresponding to the path we've augmented. In the second, we add either two or three columns and then insert either $1 \thinspace \hat{0}$, $\hat{0} \thinspace 1$, or $\hat{0}\thinspace 1 \thinspace \hat{0}$ in the row corresponding to the generalized star we augmented. The three cases account for augmenting the beginning, end, or middle of the row respectively. By the sufficiency of the LSP, each of these lists is actually achieved by $T'$, so that $\mathcal{L}^1_0(T) \subseteq \mathcal{L}_0(T').$

Consider a multiplicity list $L' \in \mathcal{L}_0(T')$. If $T'$ was obtained from $T$ by augmenting an intermediate path, we simply delete a $1$ from the row corresponding to that path, which yields a valid superposition $L \in \mathcal{L}_0(T)$. The effect of adding this $1$ back in is to augment some element of $L$ by $1$, hence $L' \in \mathcal{L}^+_0(T)$. If we obtained $T'$ from $T$ by augmenting an arm in a generalized star, we can use the same reasoning but have multiple cases. In the row corresponding to the generalized star we augmented, either we remove $1 \thinspace \hat{0}$, $\hat{0} \thinspace 1, \hat{0}\thinspace 1 \thinspace \hat{0}$, or map $\widehat{m}\to\widehat{m}-1$ for some $m\geq 1$. This gives $L\in \mathcal{L}_0(T)$, and recovers $L'$ by augmenting some element of $L$ by $1$. Since these exhaust the possibilities for generating $T'$ from $T$, the claim follows.
\end{proof}

There are several basic functions studying extremal multiplicity lists that occur for a given tree. One such function is $U(T)$, which we let denote the minimum number of ones among the multiplicity lists of $T$. Note that $U(T)\geq 2$ for any tree with at least two vertices, since the largest and smallest eigenvalues of a matrix whose graph is a tree must occur with multiplicity 1 \cite{Johnson3}.
\begin{cor}
For any linear tree $T$, we have $$ U(T) \leq 2+D_2(T),$$
where $D_2(T)$ denotes the number of degree two vertices of $T$.
\end{cor}
\begin{proof}
The paper \cite{Johnson2} shows that this upper bound follows for any tree satisfying the Degree Conjecture.
\end{proof}
This is not a strict inequality, since a star on more than three vertices has no degree two vertices and achieves $U(T)=2$.

Another function studying extremal multiplicity lists is $M(T)$, which counts the maximum multiplicity of an eigenvalue in all multiplicity lists achievable by $T$. The following corollary is a direct generalization of the known statement \cite{Johnson3} $$M( LT(T_1,s,T_2)) = M(T_1) +M(T_2)+1$$ for $s\geq 1$. Contrast it with the easily shown equality $$M\left(\oplus_{i=1}^k T_i\right) =\sum_{i=1}^k M(T_i),$$
where each $T_i$ is a generalized star. 

\begin{cor}
For any linear tree $T=LT(T_1,s_1,\ldots, s_{k-1},T_k)$, let $l$ denote the number of non-empty intermediate paths $s_i$. Then we have $$M(T) =  l+\sum_{i=1}^k M(T_i). $$
\end{cor}
\begin{proof}
We can construct a column in the tabular form of the LSP which achieves an eigenvalue with this maximum multiplicity. Begin by fixing an eigenvalue $\lambda$, and then consider the column corresponding to it. Each of the $k$ $b_i^+$ rows has the maximal element $M(T_i)$, and each of the $l$ non-empty $c_i^+$ rows has maximum element $1$. We can align these local maxima under the column corresponding to $\lambda$, and then fill in the table to the left and right consistently (for instance, by having a single element in each column, while each row retains its original order). Because the contribution by each row is locally maximum, this is a global bound. By the sufficiency of the LSP conditions a matrix for $L$ achieving this bound actually exists, giving equality.
\end{proof}

\section{More Consequences: Diameter Minimality of Linear Trees}\label{sec5}
Closely related to $U(T)$ is $c(T)$, the minimum number of distinct eigenvalues that a matrix whose graph is $T$ must have. If $T$ is a path, $c(T)=n$, and if $T$ is a star we have $c(T)=3$. We have the bound $c(T)\geq d(T)$ \cite{Johnson3}, where $d(T)$ is the diameter of $T$, the length of the longest path between two vertices in $T$, measured in terms of vertices. If we have the equality $c(T)=d(T)$, we call $T$ \textit{diminimal} (for diameter minimal). While $c(T)\neq d(T)$ in general, every tree on $\leq 6$ vertices is diminimal. We demonstrate an analogous statement for linear trees.

A generalized star $T$ with arm lengths $(l_1,l_2,\ldots)$ (in decreasing order) has diameter $d(T)=l_1+l_2+1$, and is also diminimal. When $T$ has only two arms, we can associate a canonical multiplicity list with length $d(T)$ to $T$, which we call an \textit{optimal list}. The optimal list equals $(1,\hat{1},1,\ldots,1,\hat{1},1,\hat{0},1,\ldots,1,\hat{0},1)$, where we have $l_2$ upward $\hat{1}$'s and $l_1-l_2$ upward $\hat{0}$'s, so that the total list has length $2l_1+1$. Note than this is a valid multiplicity list, and contains $2l_1+1-(l_1-l_2)=l_1+l_2+1 = d(T)$ nonzero entries. An optimal list for a path of length $n$ will be $(1,0,1,0,\ldots,0,1)$, which has $n-1$ zeros and $n$ ones. 

\begin{thm}\cite{Buckley, Johnson3}\label{dimin}
Every linear tree $T$ is diminimal, i.e., $c(T)=d(T)$.
\end{thm}
\begin{proof}
Since we already know $c(T)\geq d(T)$ for any tree, we simply need to construct a multiplicity list with length $d(T)$ to prove diminimality. 

We first reduce the problem to considering linear trees where each generalized star has $\leq 2$ arms. Given a linear tree $T = LT(T_1,s_1,\ldots,s_{k-1},T_k)$, we construct the associated tree $T' =  LT(T_1',s_1,\ldots,s_{k-1},T_k')$, where  $T_i'$ is obtained from $T_i$ by deleting every arm except the two longest. Since $d(T)$ can only depend on the two longest arms across every star and the intermediate paths, $d(T)=d(T')$. Given a multiplicity list $L\in\mathcal{L}_0(T')$ of length $d(T')$, we can construct a multiplicity list of the same length in $\mathcal{L}_0(T)$ by increasing some of the upward eigenvalues from each row corresponding to a generalized star, so that proving diminimality of $T'$ proves diminimality of $T$. 

We now consider two cases, depending on whether the diameter consists of two arms from the same generalized star, or from one arm from each of two stars and the induced connecting path. Note that some linear trees will satisfy both cases, such as the tree of Figure \ref{diam1}.
\begin{enumerate}
\item Consider the case when the diameter of the linear tree $T=LT(T_1,s_1,\ldots,s_{k-1},T_k)$ consists of two arms from the same generalized star. Construct optimal lists for every star and path component, and superpose our optimal lists as in Figure \ref{super1}, where the row in blue corresponds to the generalized star whose arms form the path inducing the diameter of $T$.

\begin{center}
\begin{figure}[H]
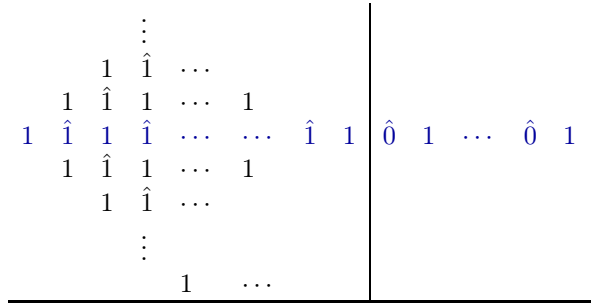

\begin{tabular}{llllllll|lllll}
                         &                                  &                          & $\vdots$                         &                                 &                                 &                                  &                          &                                  &                          &                                 &                                  &                          \\
                         &                                  & 1                        & $\hat{1}$                        & $\cdots$                        &                                 &                                  &                          &                                  &                          &                                 &                                  &                          \\
                         & 1                                & $\hat{1}$                & 1                                & $\cdots$                        & 1                               &                                  &                          &                                  &                          &                                 &                                  &                          \\
{\color[HTML]{00009B} 1} & {\color[HTML]{00009B} $\hat{1}$} & {\color[HTML]{00009B} 1} & {\color[HTML]{00009B} $\hat{1}$} & {\color[HTML]{00009B} $\cdots$} & {\color[HTML]{00009B} $\cdots$} & {\color[HTML]{00009B} $\hat{1}$} & {\color[HTML]{00009B} 1} & {\color[HTML]{00009B} $\hat{0}$} & {\color[HTML]{00009B} 1} & {\color[HTML]{00009B} $\cdots$} & {\color[HTML]{00009B} $\hat{0}$} & {\color[HTML]{00009B} 1} \\
                         & 1                                & $\hat{1}$                & 1                                & $\cdots$                        & 1                               &                                  &                          &                                  &                          &                                 &                                  &                          \\
                         &                                  & 1                        & $\hat{1}$                        & $\cdots$                        &                                 &                                  &                          &                                  &                          &                                 &                                  &                          \\
                         &                                  &                          & $\vdots$                         &                                 &                                 &                                  &                          &                                  &                          &                                 &                                  &                          \\
                         &                                  &                          &                                  & 1                               & $\cdots$                        &                                  &                          &                                  &                          &                                 &                                  &                          \\ \hline
\end{tabular}
\caption{The superposition in Case 1}
\label{super1}
\end{figure}
\end{center}

For a concrete example, consider the tree of Figure \ref{diam1}. Note that our construction forms a valid superposition since each column alternates upward and non-upward multiplicities. Furthermore, the resulting list still has diameter $d(T)$ since we only augment nonzero entries from the optimal list for $T_i$. If we are $j$ rows below (or above) the longest row (which corresponds to $T_i$), there are still $l_2-j$ upward multiplicities which are at the bottom (or top) of a column. A generalized star at this row must have a longest arm of length $\leq l_2-j$, else we could find a larger diameter which includes a portion of this star. Similarly, any path at this row must have $\leq l_2-j$ vertices. Therefore, the optimal list will always terminate to the left of the indicated line, and our superposed list will have as many nonzero entries as the longest row, which is exactly $d(T)$!

\begin{figure}[H]
\centering
\def\r{1.5}
\def\k{6pt}
\begin{tikzpicture}[scale=0.7,
bor/.style={circle, draw, only marks, mark=*, fill=white,mark size=\k},
rel/.style={rectangle,draw,rounded corners=0.6ex,  fill=white, minimum size=\k},
coo/.style={circle, draw, only marks, mark=*, fill=blue, minimum size=\k+3pt},
]
\draw (-\r,0) node[bor]{}-- (0,0) node[coo]{}-- (0,-\r)node[coo]{}-- (0,-2*\r)node[coo]{}-- (0,3*\r)node[coo]{} --(0,\r)node[coo]{} --(0,2*\r)node[coo]{} --(0,0) -- (\r,0)node[bor]{};
\foreach \x/\xtext in {2,6}{
\draw[shift={(-\r,0)}] (0pt,0pt) node[bor]{} -- ({\r*cos(\x*pi/4 r)}, {\r*sin(\x*pi/4 r)})
node[bor]{};}
\foreach \x/\xtext in {-2,2}{
\draw[shift={(\r,0)}] (0pt,0pt) node[bor]{}-- ({\r*cos(\x*pi/4 r)}, {\r*sin(\x*pi/4 r)})
node[bor]{};
}
\end{tikzpicture}
\hspace{2cm}
\begin{tabular}{lllll|ll}
  & 1         & $\hat{1}$ & 1         &   &           &   \\
{\color{blue}1} & {\color{blue}$\hat{1}$} & {\color{blue}1}         & {\color{blue}$\hat{1}$}& {\color{blue}1} & {\color{blue}$\hat{0}$} & {\color{blue}1} \\
  & 1         & $\hat{1}$ & 1         &   &           &   \\ \hline
1 & 3         & 3        & 3         & 1 & 0         & 1
\end{tabular}
\caption{An example of Case 1 (diameter in blue)}
\label{diam1}
\end{figure}
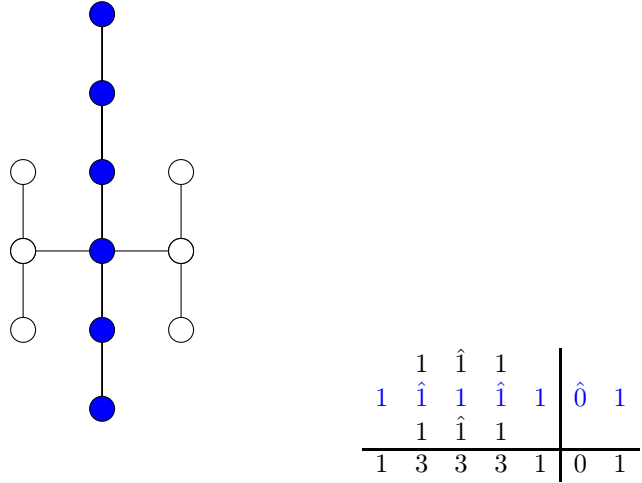

\item Consider the case where the diameter of our linear tree includes components from two generalized stars, and the induced path between them. Let the generalized star with the longest arm in the entire linear tree have arms labelled by $l_1$ and $l_2$, with $l_1\geq l_2$. Let the other generalized star with an arm contributing to the diameter have arms labelled by $m_1$ and $m_2$, with $m_1\geq m_2$. If the induced intermediate path has length $n$, then the diameter of our tree is $n+l_1+m_1+2$. 

As before, we construct optimal lists for each star and path component of our linear tree. We then superpose these optimal lists as in Figure \ref{fig53}, where the blue rows correspond to generalized stars contributing to the diameter, and the green rows correspond to the intermediate path. The vertical line denotes the point after which both blue lists do not contain any $\hat{1}$'s. After superposing the blue and green rows, the multiplicity list will have length $\leq 2+ n+l_1+l_2 + (m_1 - l_2)=d(T)$, since each of the $n$ green rows adds one to the length of the multiplicity list, while leaving the number of upward eigenvalues we can superpose with unchanged. Therefore, at each step we are still able to superpose, and superposing the last row augments at most $m_1-l_2$ columns which originally summed to $0$. Superposing the black rows above and below will not increase the length of the multiplicity by identical reasoning to the previous case; any row which extends past the vertical black line implies the existence of a sufficiently long branch  contradicting the maximality of the diameter. Figure \ref{diam2} gives an explicit example.

\begin{figure}[H]
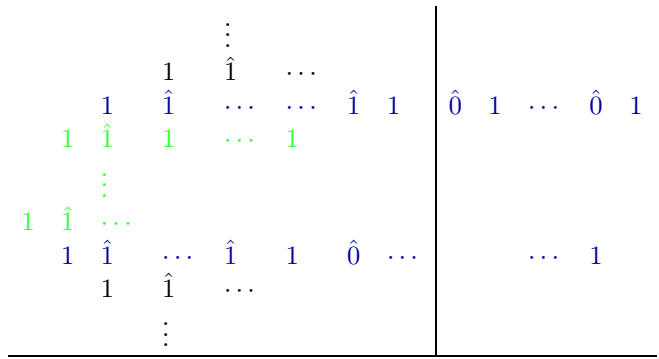

\begin{tabular}{llllllll|lllll}
                         &                                  &                                  &                                  & $\vdots$                         &                                 &                                  &                                 &                                  &                          &                                 &                                  &                          \\
                         &                                  &                                  & 1                                & $\hat{1}$                        & $\cdots$                        &                                  &                                 &                                  &                          &                                 &                                  &                          \\
{\color[HTML]{00009B} }  & {\color[HTML]{00009B} }          & {\color[HTML]{00009B} 1}         & {\color[HTML]{00009B} $\hat{1}$} & {\color[HTML]{00009B} $\cdots$}  & {\color[HTML]{00009B} $\cdots$} & {\color[HTML]{00009B} $\hat{1}$} & {\color[HTML]{00009B} 1}        & {\color[HTML]{00009B} $\hat{0}$} & {\color[HTML]{00009B} 1} & {\color[HTML]{00009B} $\cdots$} & {\color[HTML]{00009B} $\hat{0}$} & {\color[HTML]{00009B} 1} \\
{\color[HTML]{34FF34} }  & {\color[HTML]{34FF34} 1}         & {\color[HTML]{34FF34} $\hat{1}$} & {\color[HTML]{34FF34} 1}         & {\color[HTML]{34FF34} $\cdots$}  & {\color[HTML]{34FF34} 1}        & {\color[HTML]{34FF34} }          & {\color[HTML]{34FF34} }         & {\color[HTML]{34FF34} }          & {\color[HTML]{34FF34} }  & {\color[HTML]{34FF34} }         & {\color[HTML]{34FF34} }          & {\color[HTML]{34FF34} }  \\
{\color[HTML]{34FF34} }  & {\color[HTML]{34FF34} }          & {\color[HTML]{34FF34} $\vdots$}  & {\color[HTML]{34FF34} }          & {\color[HTML]{34FF34} }          & {\color[HTML]{34FF34} }         & {\color[HTML]{34FF34} }          & {\color[HTML]{34FF34} }         & {\color[HTML]{34FF34} }          & {\color[HTML]{34FF34} }  & {\color[HTML]{34FF34} }         & {\color[HTML]{34FF34} }          & {\color[HTML]{34FF34} }  \\
{\color[HTML]{34FF34} 1} & {\color[HTML]{34FF34} $\hat{1}$} & {\color[HTML]{34FF34} $\cdots$}  & {\color[HTML]{34FF34} }          & {\color[HTML]{34FF34} }          & {\color[HTML]{34FF34} }         & {\color[HTML]{34FF34} }          & {\color[HTML]{34FF34} }         & {\color[HTML]{34FF34} }          & {\color[HTML]{34FF34} }  & {\color[HTML]{34FF34} }         & {\color[HTML]{34FF34} }          & {\color[HTML]{34FF34} }  \\
{\color[HTML]{00009B} }  & {\color[HTML]{00009B} 1}         & {\color[HTML]{00009B} $\hat{1}$} & {\color[HTML]{00009B} $\cdots$}  & {\color[HTML]{00009B} $\hat{1}$} & {\color[HTML]{00009B} 1}        & {\color[HTML]{00009B} $\hat{0}$} & {\color[HTML]{00009B} $\cdots$} & {\color[HTML]{00009B} }          & {\color[HTML]{00009B} }  & {\color[HTML]{00009B} $\cdots$} & {\color[HTML]{00009B} 1}         & {\color[HTML]{00009B} }  \\
                         &                                  & 1                                & $\hat{1}$                        & $\cdots$                         &                                 &                                  &                                 &                                  &                          &                                 &                                  &                          \\
                         &                                  &                                  & $\vdots$                         &                                  &                                 &                                  &                                 &                                  &                          &                                 &                                  &                          \\ \hline
\end{tabular}
\caption{The superposition of Case 2}
\label{fig53}
\end{figure}

\begin{figure}[H]
\centering
\def\r{1.5}
\def\k{6pt}
\begin{tikzpicture}[scale=0.7,
bor/.style={circle, draw, only marks, mark=*, fill=white,mark size=\k},
rel/.style={circle, draw, only marks, mark=*, fill=green, minimum size=\k+3pt},
coo/.style={circle, draw, only marks, mark=*, fill=blue, minimum size=\k+3pt},
]
\draw (-\r,0) node[bor]{}-- (0,0) node[coo]{}-- (0,-\r)node[bor]{}  --(0,\r)node[coo]{} --(0,2*\r)node[coo]{} --(0,0) -- (\r,0)node[rel]{} --(2*\r,0)node[bor]{};
\foreach \x/\xtext in {2,6}{
\draw[shift={(-\r,0)}] (0pt,0pt) node[bor]{} -- ({\r*cos(\x*pi/4 r)}, {\r*sin(\x*pi/4 r)})
node[bor]{};}
\foreach \x/\xtext in {-2}{
\draw[shift={(2*\r,0)}] (0pt,0pt) node[coo]{}-- ({\r*cos(\x*pi/4 r)}, {\r*sin(\x*pi/4 r)})
node[bor]{};
}
\foreach \x/\xtext in {2}{
\draw[shift={(2*\r,0)}] (0pt,0pt) node[coo]{}-- ({\r*cos(\x*pi/4 r)}, {\r*sin(\x*pi/4 r)})
node[coo]{};
}
\end{tikzpicture}
\hspace{2cm}
\begin{tabular}{lllll|l}
                         &                          & 1                                & $\hat{1}$                        & 1                                &                          \\
{\color[HTML]{00009B} }  & {\color[HTML]{00009B} 1} & {\color[HTML]{00009B} $\hat{1}$} & {\color[HTML]{00009B} 1}         & {\color[HTML]{00009B} $\hat{0}$} & {\color[HTML]{00009B} 1} \\
{\color[HTML]{34FF34} 1} & {\color[HTML]{34FF34} }  & {\color[HTML]{34FF34} }          & {\color[HTML]{34FF34} }          & {\color[HTML]{34FF34} }          & {\color[HTML]{34FF34} }  \\
{\color[HTML]{00009B} }  & {\color[HTML]{00009B} }  & {\color[HTML]{00009B} 1}         & {\color[HTML]{00009B} $\hat{1}$} & {\color[HTML]{00009B} 1}         & {\color[HTML]{00009B} }  \\ \hline
1                        & 1                        & 3                                & 3                                & 2                                & 1                       
\end{tabular}
\caption{An example of Case 2 (diameter in blue and induced path in green)}
\label{diam2}
\end{figure}
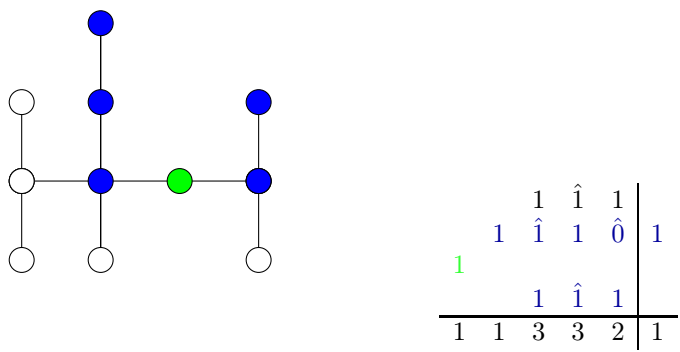
\end{enumerate}

\end{proof}

\section*{Acknowledgements}
The authors gratefully acknowledge support from 2018 National Science Foundation grant DMS \#0751964 that supported the work reported herein. Tanay Wakhare would also like to thank Charles Johnson and Roberto Costas-Santos for providing a very supportive environment at the 2018 William \& Mary Matrix Analysis REU, where this work was conducted.

\end{document}